\documentclass[11pt,a4paper]{amsart}
\usepackage{color}
\usepackage{textcomp}
\usepackage{mathtools}
\usepackage{amstext}
\usepackage{amsthm}
\usepackage{amssymb}
\usepackage{setspace}
\usepackage{fullpage}
\newtheorem{thm}{Theorem}[section] %the resolution could also be [subsection]

\newtheorem*{claim*}{Claim}

\newtheorem{cor}[thm]{Corollary}
\newtheorem{defn}[thm]{Definition}

\newtheorem{lem}[thm]{Lemma}
\newtheorem{prop}[thm]{Proposition}

\newtheorem{rem}[thm]{Remark}

\newtheorem{ques}[thm]{Question}

\newtheorem{obs}[thm]{Observation}
\newtheorem{thmx}{Theorem}
\def\Qp{\mathbb Q_p}
\def\Zp{\mathbb{Z}_p}
\def\Ind{\mathrm{Ind}}
\def\ind{\mathrm{ind}}
\def\End{\mathrm{End}}
\def\Ext{\mathrm{Ext}}
\def\Hom{\mathrm{Hom}}
\def\complex{\mathbb{C}}
\def\reals{\mathbb{R}}
\def\nats{\mathbb{N}}
\def\ints{\mathbb{Z}}
\def\rats{\mathbb{Q}}
\def\supp{\mathrm{supp}}
\def\Dist{\mathrm{Dist}}

\def\repg{\mathrm{Rep}(G)}
\def\H{\mathcal{H}}

\newcommand\angles[1]{\langle #1 \rangle}
\newcommand\curly[1]{\{ #1 \}}
\newcommand\Irr[1]{\mathrm{Irr}{( #1 )}}

\AtEndDocument{\bigskip{\footnotesize
  \textsc{Yotam Hendel, Faculty of Mathematics and Computer Science, The Weizmann Institute of
Science,  Rehovot 76100, Israel.
} \\
  \textit{E-mail address}: \texttt{yotam.hendel@gmail.com} \par
  \addvspace{\medskipamount}
}}
\begin{document}
\author{Yotam I. Hendel}
\title{On twisted Gelfand pairs through commutativity of a Hecke algebra}
\maketitle
\begin{abstract}
For a locally compact, totally disconnected group $G$, a subgroup $H$ and a character $\chi:H \to \mathbb{C}^{\times}$
we define a Hecke algebra $\H_\chi$ and explore the connection between commutativity of $\H_\chi$ and the $\chi$-Gelfand property of $(G,H)$, i.e. the property $\dim_{\mathbb{C}} (\rho^*)^{(H,\chi^{-1})} \leq 1$ for every $\rho \in \mathrm{Irr}(G)$, 
the irreducible representations of $G$. 

We show that the conditions of the Gelfand-Kazhdan criterion imply commutativity of $\H_\chi$, and 
verify in several simple cases that commutativity of $\H_\chi$ is equivalent to the $\chi$-Gelfand property of $(G,H)$.

We then show that if $G$ is  a connected reductive group over a $p$-adic field $F$, 
and $G/H$ is $F$-spherical, then the cuspidal part of $\H_\chi$ is commutative if and only if  $(G,H)$ satisfies the $\chi$-Gelfand property with respect to all cuspidal representations ${\rho \in \mathrm{Irr}(G)}$.

We conclude by showing that if $(G,H)$ satisfies the $\chi$-Gelfand property 
with respect to all irreducible $(H\backslash G,\chi^{-1})$-tempered representations of $G$ then $\H_\chi$ is commutative.
\end{abstract}
\tableofcontents

\section{Introduction}
\subsection{Motivation}
\subsubsection{Gelfand pairs of finite groups} 
Let  $G$ be a finite group, and let $\Irr{G}$ denote the set of its irreducible representations.
It is a fundamental result that the group algebra $\mathbb {C}[G]$ is semi-simple, and that each 
$\rho \in \Irr{G}$
appears $\dim_\complex \rho$ times in its direct sum decomposition.
 
A question that comes to mind is what happens in the relative situation, 
i.e. given a subgroup $H \leq G$, how does the representation of $H$-invariant functions $\complex[G]^H\simeq \complex[H \backslash G]$ decompose.
Here, not all irreducible representations of $G$ appear in $\complex[H \backslash G]$, and those that do are called {\it $H$-distinguished}. 
\begin{defn}
We say that $(G,H)$ is a {\it Gelfand pair}
if every $ \rho \in \Irr{G}$ appears in $\complex[H \backslash G]$ at most once, i.e. $\dim_\complex\Hom_G(\rho,\complex[H \backslash G]) \leq 1$.
\end{defn}
Note that by Frobenius reciprocity, this is equivalent to demanding that every $\rho \in \Irr{G}$ 
has at most one $H$-invariant linear functional up to a scalar ($\dim_\complex (\rho^*)^{H}=\dim_\complex\Hom_H(\rho|_H,\complex) \leq 1$).

The notion of Gelfand pairs has found many uses throughout mathematics, and is used, for example, to study the representation theory of the symmetric group (e.g. \cite{OV96})
%Prominent examples of such uses include 
%bounding Fourier coefficients of Maass forms and the study of periods  in the theory of automorphic forms (e.g. \cite{GPSR97},\cite{Gr91},\cite{Re08}, these use Gelfand pairs of non-finite groups, as in Definition \ref{def:GelRed}),
and when investigating the convergence rates of random walks on finite groups (e.g. \cite[Chapter 3F]{Di88},\cite{Le82}).

The Gelfand property can also be viewed as a relative analogue of Schur's lemma;
indeed, in the group case, namely $\Delta G \leq G \times G$ where $\Delta G$ is a diagonal copy of $G$ embedded in $G \times G$, the Gelfand property of the pair $(G \times G ,\Delta G)$ is a restatement of Schur's lemma.

In principle, in order to verify the Gelfand property of a pair $(G,H)$, one has to calculate the dimension of the space of $H$-invariant functionals on every 
$\rho \in \mathrm{Irr}(G)$,
and show that in each case it is at most $1$.
This is a hard task even in the case of finite groups, 
but conveniently
the Gelfand property can be formulated in terms of commutativity of the convolution algebra of bi-$H$-invariant functions on $G$, which in the case of finite groups is isomorphic to $\End_G(\complex[H \backslash G])$.
This simple observation is key to the work presented in this paper:
\begin{obs} \label{obs:HeckeCom}
Let $H\leq G$ be finite groups. 
 $(G,H)$ is a Gelfand pair $\iff$ the Hecke algebra $\mathbb {C}[ G ]^{H\times H}$ 
% \simeq \End_G(\complex[H \backslash G])$
  is commutative.
\end{obs}
Recall that an anti-involution $\sigma: G \to G$ is a map such that  $\sigma(g_1 g_2)=\sigma(g_2)\sigma(g_1) $ for all $g_1,g_2 \in G$ and $\sigma^2 = \mathrm{Id}$.
In light of Observation \ref{obs:HeckeCom}, 
we can show that a pair has the Gelfand property using only the relative group structure of $G$ with respect to $H$, without knowing anything about the irreducible representations of $G$, or the decomposition of $\complex[H \backslash G]$:
\begin{prop} [Gelfand's trick] \label{prop:Gel}
Assume there exists an anti-involution $\sigma$ of $G$ 
such that $\sigma(H)=H$ and  $H\sigma(g)H=HgH$ for all $g \in G$, then $(G,H)$ is a Gelfand pair.
\end{prop}
\subsubsection{Non-compact subgroups and the Gelfand-Kazhdan criterion} \label{subsub:CompactSubgroup}
The definition of a Gelfand pair can be generalized naturally to the case where $G$ is a locally compact group
and $H\leq G$ is a compact subgroup, by 
demanding $\dim_\complex\Hom_H(\rho|_H,\complex) \leq 1$ for every $\rho \in \mathrm{Irr}(G)$, where we either consider smooth irreducible representations of $G$ or unitary irreducible ones.
% http://hss.ulb.uni-bonn.de/2004/0513/0513.pdf
%either considering smooth or unitary representations of $G$ and
In both these settings Observation \ref{obs:HeckeCom} and
Gelfand's trick hold (see \cite[Theorems 6.1.3 and 6.3.1]{vD09}, \cite[Section 2.10]{BZ}), 
%\cite[Proposition 6.1.3]{vDi09}
%\cite[Proposition 6.3.1]{vDi09}
%\cite[Page 53, Theorem 1]{Lang}
where 
$\mathbb {C}[ G ]^{H\times H}$ is replaced by $C_c(G)^{H\times H}$,
the algebra of compactly supported, bi-$H$-invariant continuous 
{functions} on $G$.

Let $G$ be a locally compact, totally disconnected topological
group and let $H \leq G$ be a closed subgroup.
We consider its category of smooth representations $\repg$.
For simplicity, assume {both $G$ and $H$ are unimodular}.
We would like to generalize the Gelfand property to these settings.
While the straight-forward definition would be demanding 
$\dim_\complex \Hom_H(\rho|_H, \complex) \leq 1$
for
every $\rho \in \Irr{G}$ 
(GP1 as below), 
it is often much easier 
to prove {an (a-priori)} weaker statement (GP2 as below).
There are no known examples of pairs $(G,H)$ which satisfy GP2 but not GP1, and it is conjectured that these two conditions are equivalent in general.
\begin{defn} [In spirit of {\cite[Definition 2.2.1]{AGS08}}]
\label{def:GelRed}
Let $H\leq G$ be as above.
\begin{enumerate}
\item We say that $(G,H)$ satisfies GP1 if $\dim_\complex \Hom_H(\rho|_H,\complex) \leq 1$ for every $\rho \in \Irr{G}$.
\item We say that $(G,H)$ satisfies GP2 if $\dim_\complex \Hom_H(\rho|_H,\complex)\cdot \dim_\complex \Hom_H(\tilde\rho|_H,\complex) \leq 1$ for every 
$\rho \in \Irr{G}$ and its smooth dual $\tilde \rho$.
\end{enumerate}
\end{defn}

\begin{rem}
The concept of Gelfand pairs in the case where $H$ non-compact have proven to be particularly important when studying periods in the theory of automorphic forms, see \cite{GPSR97},\cite{Gr91} or \cite{Re08}.
%Gelfand pairs $(G,H)$ where $H$ is non-compact have proven to be of particular importance when studying periods in the theory of automorphic forms, e.g. \cite{GPSR97},\cite{Gr91},\cite{Re08}.
\end{rem}

Let $C^{-\infty}(G)$ denote the space of generalized functions on $G$, 
i.e. the dual of the space of all locally constant, compactly supported measures on $G$.
The main mechanism used to show the Gelfand property (in the sense of GP2) in these settings is the following generalization 
of Gelfand's trick (see  \cite{GK75}, \cite[Lemma 4.2]{P90} or \cite[Proposition 4.2]{Gr91}).
\begin{thm} [Gelfand-Kazhdan criterion]\label{prop:GK} 
Assume there exists an anti-involution $\sigma:G\to G$ such that $\sigma(H)=H$ and $\sigma(\xi)=\xi$ for every generalized function $\xi \in C^{-\infty}(G)^{H \times H}$,
then $(G,H)$ is a Gelfand pair (in the sense of GP2).
\end{thm} 

When comparing the above to the case where $H$ is compact, we see that 
the space $C_c(G)^{H \times H}$ is
replaced by the space of invariant generalized functions, which is not an algebra.
Furthermore, it is not clear what is the analogue of Observation \ref{obs:HeckeCom}, if it exists.
Evidently, if such an algebra existed, validity of the conditions of the Gelfand-Kazhdan criterion would imply it is commutative.
We arrive at the following question.

\begin{ques} \label{ques:GKRes}
Can one define a Hecke algebra $\mathcal H$, analogous to  {$C_c(G)^{H \times H}$}, such that
\begin{enumerate}
\item The Gelfand-Kazhdan conditions imply commutativity of $\mathcal H$.
\item $\mathcal H$ is commutative if and only if  $(G,H)$ is a Gelfand pair.
\end{enumerate}
\end{ques}

\subsection{Summary of the main results}
Let $G$ be a locally compact, totally disconnected topological group, 
let $H \leq G$ be a closed subgroup, and let $\chi: H \to \complex^{\times}$ be a smooth character.
{For simplicity, assume $(G,H)$ is a unimodular pair}.
We say that $(G,H)$ is a {\it $\chi$-Gelfand pair} if GP1 holds where $\complex$ is replaced by $\complex_\chi$, i.e. 
$d_{H,\chi}(\rho):=\dim_\complex (\rho^*)^{(H,\chi^{-1})} \leq 1$ for all $\rho \in \mathrm{Irr}(G)$.

Let $\ind_H^G$ denote the compact induction functor. In this paper we define a Hecke algebra $\mathcal{H}_\chi:= \End_G (\ind_H^G \chi^{-1})$, which can conveniently be considered as an algebra of invariant distributions,
 and give a partial answer to Question \ref{ques:GKRes} as follows.
We first answer \ref{ques:GKRes}(1) affirmatively (set $\H:=\H_1$, where $1$ is the trivial character of $H$):
\begin{thmx}[See Theorem \ref{thm:HeckeCommutative} for a more general result]
Assume the Gelfand-Kazhdan conditions hold (Proposition \ref{prop:GK}), i.e. there exists an anti-involution $\sigma: G \to G$ such that 
$\sigma(H)=H$ and $\sigma(\xi)=\xi$ for every $\xi \in C^{-\infty}(G)^{H \times H}$, then $\H$ is commutative.
\end{thmx}
We then address  \ref{ques:GKRes}(2) in several cases where $H \backslash G$ is especially well behaved:
\begin{thmx}[For exact conditions on $\chi$ see Corollary \ref{cor: GP1 for compact subgroup}, Proposition \ref{GP1 for G/H compact} and Proposition \ref{cor:HeckeCase}]~
\label{thm:geometric cases}
\begin{enumerate}
\item  If $H $ is either compact or co-compact, then $\H_\chi$ is commutative $\iff$ $(G,H)$  is a $\chi$-Gelfand pair.
\item If $H $ is open and commensurated in $G$ (i.e. $gHg^{-1}\cap H$ has finite index in $H$ for all $g \in G$),
then  $\H_\chi$ is commutative if and only if 
$d_{H,\chi}(\rho)\leq 1$ 
 for every $\rho \in \Irr{G}$  such that either $\rho^{H_S}\neq \curly{0}$
 for a finite set $S \subset G/H$ 
  where $H_S = H \bigcap\limits_{g \in S} g H g^{-1}$, or $\rho$ is admissible.
\end{enumerate}
\end{thmx}

In particular, we see that if $H$ is compact then $\H$ is commutative $\iff$ $C_c(G )^{H \times H}$
is commutative, so it is sensible to regard $\H$ as a generalization of $C_c(G)^{H \times H}$.

Furthermore, using Theorem \ref{thm:geometric cases}(2) we show that given a reductive algebraic group ${\bf G}$ satisfying the strong approximation property (see Theorem \ref{thm:strong approx}),  the pair of discrete groups $(G,H)=(\bf{G}(\rats),\bf{G}(\ints))$  satisfies $d_{H,1}(\rho)\leq 1$  for every smooth irreducible representation $\rho$ of $G$ such that $\rho^{H_S}\neq \curly{0}$ for some finite set $S \subset G/H$ and $H_S$ as above
%a Gelfand property as in \ref{thm:geometric cases}(2) above 
(note $\bf{G}(\ints)$ is open and commensurated in $\bf{G}(\rats)$).
This is done by translating the Gelfand property of a classical pair to the settings of $(\bf{G}(\rats),\bf{G}(\ints))$ via the Schlichting completion (see Section \ref{sec:Hecke pairs}).

For the next result, 
assume $G={\bf G}(F)$ and $H={\bf H}(F)$,
where ${\bf G}$ is a connected reductive group, ${\bf H}\leq {\bf G}$ a Zariski closed subgroup,
and $F$ is a $p$-adic field.
 We also assume $(G,H)$ is an $F$-spherical pair (see Definition \ref{def:sphericalpair}).

It is conjectured that $F$-sphericity implies 
$d_{H,\psi}(\rho)=\dim_\complex\mathrm{Hom}_G(\ind_H^G \psi^{-1},\tilde{\rho} )< 
\infty $ for every admissible $\rho \in \repg$ and character $\psi$ of $H$ (this was verified in several cases - see \cite{Del10},\cite{SV17}).
In order to prove Theorem \ref{thm:MainTheorem}, it would be sufficient to assume $d_{H,\chi}(\rho)$ is finite for every irreducible cupsidal representation of $G$.

By analyzing $d_{H,\chi}(\rho)$ over the cuspidal blocks of $\repg$,
 and using \cite[Theorem 6.1]{AS} and results from \cite{AAG12} 
  we show the following:
\begin{thmx} [See Theorem \ref{thm:commHeckeImpliesCGP} {for a more general result}]
\label{thm:MainTheorem}
The cuspidal part of $\H_\chi$ is commutative $\iff$ $(G,H)$ is a cuspidal $\chi$-Gelfand pair, i.e. 
$d_{H,\chi}(\rho) \leq 1$
for every cuspidal $\rho \in \mathrm{Irr}(G)$.
\end{thmx}
\begin{rem}
Note that for a cuspidal representation $\rho\in \mathrm{Irr}(G)$, we have $d_{H,\chi}(\rho)=d_{H,\chi^{-1}}(\tilde{\rho})$, so the conditions GP1 and GP2 (as in Definition \ref{def:GelRed}) are equivalent when restricting to the class of irreducible cuspidal representations of $G$.
\end{rem}
The last result of this paper concerns the converse direction of \ref{ques:GKRes}(2):
\begin{thmx} [See Theorem \ref{thm:GPimpliesCom} {for a more general result}]
Let $G$ be a second countable group of type I, and let $\chi$ be a unitary character.
If 
$d_{H,\chi}(\rho) \leq 1$
for every  $(H\backslash G,\chi^{-1})$-tempered $\rho \in \mathrm{Irr}(G)$ (i.e. included in the support of the Plancherel measure of $L^2(H\backslash G,\chi^{-1})$), then $\H_\chi$ is commutative.
In particular, if $(G,H)$ is a $\chi$-Gelfand pair, then $\H_\chi$ is commutative.
\end{thmx}
\subsection{Structure of the paper}
The paper is structured as follows:
In Section 2  we review  facts from the representation theory of $p$-adic groups, the Gelfand-Kazhdan criterion, 
 the Bernstein decomposition and the direct integral decomposition of unitary representations.
In Section 3 we prove Theorem A.
In Section 4 we prove Theorem B.
In Section 5 we prove Theorem C.
In Section 6 we prove Theorem D.
In Appendix A we present a proof to a general version (i.e. twisted, non-unimodular) of the Gelfand-Kazhdan criterion.
\subsection{Related work}
{
The dimensions $d_{H,\chi}(\rho)$ (where $\rho$ ranges over $\Irr{G}$) and the Gelfand property have been studied vastly 
(e.g. \cite{Sha74}, \cite{GK75}, \cite{P90}, \cite{Hak03}, \cite{AGS08}, \cite{HM08}, \cite{Del10}, \cite{SZ11}, \cite{AG09}, \cite{AGRS10}, \cite{AAG12}, \cite{SV17})
where in almost all cases proving that a pair is a Gelfand pair (where $H\leq G$ is non-compact) passes through the Gelfand-Kazhdan criterion.

The Gelfand-Kazhdan criterion was first introduced in \cite{GK75}, where it was used to show uniqueness of Whittaker models in the case of $\mathrm{GL}_n$ over a non-Archimedean local field. }

In \cite{SZ11}, the autors give a general formulation of the Gelfand-Kazhdan criterion for real reductive Lie groups.

{In \cite{Hak03} the author shows that for a $p$-adic symmetric pair $(G,H,\theta)$,
it is enough to demand that $ZH\theta(g)H=ZHg^{-1}H$ for almost every 
 double coset (i.e. $\curly{ g\in G: ZH\theta(g)H\neq ZHg^{-1}H}$ has measure zero in $G$) in order for $(G,H)$ to be a cuspidal Gelfand pair. It will be interesting to see whether this condition implies commutativity of our $\H_\chi$.}

\subsection{Acknowledgments}
{I wish to thank {\bf Uri Bader}, {\bf Roman Bezrukavnikov}, {\bf Shachar Carmeli}, {\bf Max Gurevich}, {\bf Dmitry Gourevitch}, {\bf Erez Lapid}, {\bf Omer Offen}, {\bf Eitan Sayag} and {\bf Michael Schein}
 for various helpful discussions.
 Special thanks are due to {\bf Gil Goffer} and {\bf Waltraud Lederle} for their help with Section \ref{sec:Hecke pairs} and to 
 {\bf Itay Glazer} for numerous helpful discussions 
and for reading a preliminary version of this paper.
I also wish to thank my advisors {\bf Avraham Aizenbud} and {\bf Joseph Bernstein} for their help and guidance throughout this project.}
Lastly, I wish to thank the {\bf anonymous referee} for their time and effort, and for their 
valuable suggestions and remarks.

This work has particularly benefited from my participation in the doctoral school
``Introduction to Relative Aspects in Representation Theory, Langlands Functoriality and Automorphic Forms'' at CIRM, 
and in the ``Sphericity 2016'' conference. I wish to thank the  organizers of both conferences.

This work was partially supported by ISF grants  687/13 and 249/17, BSF grant 2012247, ERC StG grant 637912 and by a Minerva foundation grant.
\subsection{Conventions and notations}
\label{sec:conventions}
Throughout this paper $F$ is a non-Archimedean local field {of characteristic zero}.
Boldface letters denote algebraic groups (such as ${\bf G}$, ${\bf P}$), 
and the corresponding non-boldface letters denote their $F$-points.
{Unless stated otherwise, $G$ is assumed to be a locally compact, totally disconnected (abbreviated l.c.t.d), unimodular, Hausdorff topological group},
 and $H$ is assumed to be a closed subgroup, not necessarily unimodular.
 $\chi: H \to \complex^{\times}$ is assumed to a character of $H$.

We write $\delta := (\delta_G \delta_H^{-1})^{\frac{1}{2}}$ where $\delta_G$ and $\delta_H$ are the modular characters of $G$ and $H$ respectively, and $\tilde{\rho}:=(\rho^*)^{\mathrm{sm}}$ for the smooth dual of a representation  $\rho$ of $G$. We also use $\rho^{H,\chi}:=\curly{v \in \rho: h\cdot v = \chi(h)v ~\forall h\in H}$.

We write $C^\infty_c(X)$ for the space of smooth (i.e. uniformly locally constant), compactly supported functions on $X$ 
and $\Dist(X)$ for the space of distributions on $X$, i.e. all linear functionals on $C^\infty_c(X)$.
For a function $f$ on $G$ and a map $\sigma : G \to G$ we write $f^\sigma=f \circ \sigma$, 
and use $\angles{\xi^\sigma,f}=\angles{\xi,f^\sigma}$ for distributions.
\section{Preliminaries}
\subsection{General representation theoretic facts and the Gelfand-Kazhdan criterion}
Recall we consider the category $\repg$ of smooth representation of a group $G$, where $\Irr{G}$ denotes the set of its irreducible objects (up to isomorphism).
Let $\pi$ be a representation of $H$ and let $\chi : H \to \complex^{\times}$ be a character of $H$. 
\begin{defn} We define the induction and compact induction functors as follows.
\begin{enumerate}
\item $\Ind_H^G \pi := \curly{ f : G \to \complex : f(hg)=\pi(h)(f(g))\text{ and } f\text{ is smooth}}$ where the action is 
$g'\cdot f(g)= f(gg')$.
\item $\ind_H^G \pi := \curly{ f \in  \Ind_H^G \pi : \supp(f) \text{ is compact modulo }$H$}$.
\end{enumerate}
\end{defn}
Lemmas \ref{lem:standardfacts} and \ref{lem:distquotisom} are standard, and are used extensively throughout the paper.
\begin{lem} \label{lem:standardfacts}
Let $\rho$ be an irreducible representation of $G$.
\begin{enumerate}
\item $\widetilde{\ind_H^G \pi} \simeq \Ind_H^G  (\widetilde{\pi} \otimes \delta_{H \backslash G} )$.
\item $\Hom_G(\ind_H^G \chi^{-1} \delta,\rho) \simeq  \Hom_G(\tilde\rho, \Ind_H^G \delta \chi)$ if $\rho$ is admissible.
\item $\Hom_G(\rho, \Ind_H^G \pi) \simeq \Hom_H(\rho_{|H}, \pi)$.
\item $\Hom_G( \ind_H^G \pi, \rho) \simeq \Hom_H(\pi, \rho_{|H})$ if $H$ is open.
\end{enumerate}
\end{lem}
\begin{lem} \label{lem:distquotisom}
$(\ind_H^G \chi)^* \simeq \mathrm{Dist}(G)^{H,\chi\delta_H}$.
\end{lem}
\begin{proof}
Define a map $\Phi: C_c^\infty(G) \to \ind_H^G \chi$ by $\Phi_{f,\chi} (g)=\int_H f(hg)  \chi^{-1}(h) d\nu_H $.
The image satisfies the necessary equivariance conditions, and its dual gives the desired isomorphism (see \cite[Lemma 3.1]{Of}).
\end{proof}
Set $d_{H,\chi}(\rho):=\dim_\complex\Hom_H(\rho|_H , \chi)$.
{The following generalizes Definition \ref{def:GelRed}(1)} to non-unimodular pairs. 
\begin{defn} \label{def:TwistedGelfand}
We say that $(G,H)$ is a $\chi$-Gelfand pair if $\dim_\complex\Hom_H(\rho|_H , \chi{\delta}) \leq 1$ 
for every irreducible representation $\rho$ of $G$. 
\end{defn}
Recall that an involution $\mu : G \to G$ is an automorphism of $G$ of order $2$, 
and that an anti-involution $\sigma$ is a map of the form $\sigma(g)=\mu(g)^{-1}$ where $\mu$ is an involution.
The following is immediate.
\begin{lem} \label{lem:MeasureUnderInv} 
Let $\mu$ be an involution of $G$, and let $\nu_G$ be a left invariant Haar measure on $G$.
Then changing variables $x \mapsto \mu(x)$ preserves $\nu_G$, that is $\mu_*(\nu_G) = \nu_G$. 
\end{lem}
%\begin{proof}
%%Remove for paper version.
%{Note that $\mu_*(\nu_G)$ is again a left invariant Haar measure and thus $\mu_*(\nu_G)=\lambda_{\nu_G} \nu_G$ where $\lambda_{\nu_G} \in\
%\mathbb{R}_{>0}$. 
%Since $\nu_G=\mu^2_*(\nu_G)=\lambda_{\nu_G}^2 \nu_G$, where  $\lambda_{\nu_G}$ is positive, we must have $\lambda_{\nu_G}=1$.}
%\end{proof}
We now state a general version of the Gelfand-Kazhdan criterion for future reference (for a proof see Appendix \ref{app: general proof of GK}).
\begin{thm} (Gelfand-Kazhdan criterion)\label{prop:GK1}  
Assume there exists an anti-involution $\sigma:G\to G$ such that $\sigma(H)=H$ and such that for every $\xi\in C^{-\infty}(G)^{(H,\chi^{-1}) \times (H,\chi^\sigma)}$,
i.e. $\xi$ where
\[
L_h(\xi)=\chi^{-1}(h)\xi, ~~~~~~~ R_h(\xi)=\chi^\sigma(h)\xi
\]
we have $\xi^\sigma=\xi$.
Then 
$d_{H,\chi}(\rho) \cdot d_{H,\chi^\mu}(\tilde{\rho}) \leq 1$
for 
every 
$\rho\in \mathrm{Irr}(G)$, where $\mu(g)=\sigma(g)^{-1}$.
\end{thm}
\subsection{The Bernstein decomposition} \label{sec:Bdecomposition}
Let ${\bf G}$ be a reductive algebraic group and set $G={\bf G}(F)$. 
The theory of the Bernstein center allows us to study $\repg$ 
by decomposing it into smaller indecomposable sub-categories called Bernstein blocks. 
We give a short review of the parts of this theory which are used in this work,  see \cite{BD84}, \cite{BR}, or  \cite{Ro09} for a thorough treatment. 
We first need to establish some notations.
\begin{defn}
Let $G$ be as above.
We define $G_0 \leq G$ to be the inverse image of the maximal compact subgroup of $G/[G,G]$.
Alternatively, this is the subgroup of $G$ generated by all compact subgroups.
\end{defn}
The idea is that the representation theories of $G$ and $G_0$ are closely related, while $G_0$ is simpler (i.e. it has compact center).
We get that $G_0$ is an open normal subgroup, and that $G/G_0$ is a finitely generated, discrete, abelian group (see \cite[Proposition 22]{BR}).
Evidently, this gives the set of characters $\Hom_G(G/G_0,\complex^\times)$
a structure of an algebraic torus $(\complex^\times)^l$ where $l$ is the rank of $G/G_0$.
\begin{defn}
We denote the variety $\Hom_G(G/G_0,\complex^\times)$ by $\frak{X}_G$, and call a character $\chi \in \frak{X}_G$  {\it unramified}.
\end{defn}
Recall that a representation of $G$ is {\it cuspidal} if the support of its matrix coefficients 
is compact when projected to $G/Z(G)$ where $Z(G)$ is the center of $G$.
A  {\it cuspidal datum} of $G$ is a pair $(M,\rho)$
where $M={\bf M}(F)$ is a Levi of a parabolic subgroup of ${\bf G}$ and $\rho$ an irreducible cuspidal representation of $M$. 
We define an equivalence relation on the set of cuspidal data by $(M , \rho) \sim (M', \rho')$ if 
there exists $g \in G$ such that the following holds, where $\mathrm{Int}(g)$ is the conjugation action:
\begin{enumerate}
\item $\mathrm{Int}(g) (M) = M'$;
\item $\mathrm{Int}(g)(\rho) \simeq \rho'\otimes \chi$, where $\chi \in \frak{X}_M$.
\end{enumerate}
We denote the equivalence class of $(M,\rho)$ by $[M,\rho]$.
We can now describe the Bernstein blocks:
\begin{defn}
Let $(M,\rho)$ be a cuspidal datum.
\begin{enumerate}
\item We set $\Psi(M,\rho)=i_{GM}(\ind_{M_0}^M (\rho{|_{M_0}}))$ 
where $i_{GM}$ is the normalized parabolic induction from $M$ to $G$.
\item The block $\mathcal{B}_{(M,\rho)}$ is defined to be the full subcategory of $\repg$ generated by $\Psi(M,\rho)$.
\item Let $R_{(M,\rho)}:=\End_G(\Psi(M,\rho))$, and define the functor 
$H_{(-)}^{(M,\rho)}:\repg \to  \mathrm{Mod}(R_{(M,\rho)})$ by  $H_{V}^{(M,\rho)} = \Hom_G(\Psi(M,\rho),V)$.
\end{enumerate}
\end{defn}

\begin{prop} \label{prop:description of blocks}
Let $(M,\rho)$ and $(M',\rho')$ be cuspidal data.
\begin{enumerate}
\item The categories $\mathcal{B}_{(M,\rho)}$ and $\mathcal{B}_{(M',\rho')}$ are equal (as subcategories of $\repg$) if and only if  $(M,\rho)\sim(M',\rho')$, so we may write $\mathcal{B}_{[M,\rho]}$ 
for the corresponding block. 
\item The functor ${H_{(-)}^{(M,\rho)}}_{|{\mathcal{B}_{(M,\rho)}}}:\mathcal{B}_{(M,\rho)} \to  \mathrm{Mod}(R_{(M,\rho)})$ is an equivalence of categories.
\end{enumerate}
\end{prop}
\begin{proof}
For (1) see \cite[Proposition 35]{BR}.
For (2) see \cite[Lemma 22]{BR} and \cite[Theorem 23]{BR}.
%also see Proposition 27, 
\end{proof}
Let $\Omega$ be the set of equivalence classes of cuspidal data under $\sim$.
The following theorem, due to Bernstein, describes the structure of $\repg$:
\begin{thm}\cite[Decomposition Theorem]{BR} \label{thm:BerDecomp}
The category $\repg$ admits a decomposition $\repg \simeq \prod\limits_{[M,\rho]\in \Omega} \mathcal{B}_{[M,\rho]}$.
\end{thm}
In particular, in many situations we can restrict our attention to one block at a time, which by Proposition \ref{prop:description of blocks} is equivalent to a category of modules.
The blocks  which will be important for us are the cuspidal blocks.
\subsection{Cuspidal blocks}
Fix an irreducible cuspidal representation $\rho$ of $G$, and set $\Psi(\rho) := \Psi(G,\rho)=\ind_{G_0}^G (\rho{|_{G_0}})$, a projective generator of the block $\mathcal{B}_{[G,\rho]}$, and 
$H_{\rho,(-)}:=H_{(-)}^{(G,\rho)}$ for the projection to the block $\mathcal{B}_{[G,\rho]}$.

We have an isomorphism of algebras  $B := (\mathcal{O}(\frak{X}_G),\cdot)\simeq (C_c[G/G_0],*)$ by choosing generators $G/G_0\simeq \ints^l$ and sending $f \in C_c[G/G_0]$ to $\sum\limits_{i \in \ints^l} f(i)x^i \in \complex[x_1^{\pm1},\ldots,x_l^{\pm1}]$.
Now, set $R_\rho=R_{(G,\rho)}$ and embed $B \simeq C_c[G/G_0]$ into $R_\rho$ by $(b\cdot f)(x)=b(x)f(x)$
for $f \in \Psi(\rho)$ and $b \in B$. 
Finally, set $Z_\rho=Z(R_\rho)$.

\begin{thm}[{\cite[Proposition 4.5]{AS}}] \label{thm:RamiEytanresults}
Let $Z_\rho \subseteq B \subseteq R_\rho$ be as above, then the following hold:
\begin{enumerate}
\item After applying the functor $B\otimes_{Z_\rho}$ to the triple $Z_\rho \subseteq B \subseteq R_\rho$
we get 
\[
B\cdot \mathrm{Id} \subseteq \mathrm{diag}(B) \subseteq \mathrm{M}_n(B),
\]
 where $\mathrm{M}_n(B)$ denotes the algebra of $n\times n$ matrices with values in $B$, $\mathrm{diag}(B)$ its subalgebra of 
diagonal matrices 
and $B \cdot \mathrm{Id}$ the subalgebra of scalar matrices  with values in $B$.
\item The map $\mathrm{Spec}(B) \to \mathrm{Spec}(Z_{\rho})$ induced by the inclusion $Z_{\rho} \subseteq B$ is surjective and \'{e}tale.
\end{enumerate}
\end{thm}
Next we define the property of being an $F$-spherical pair and 
state a result for such pairs, due to Aizenbud-Sayag, which is a main ingredient 
in the proof of Theorem \ref{thm:MainTheorem}.
\begin{defn}  \label{def:sphericalpair}
A pair $(G,H)$ is called $F$-spherical if 
$|H\backslash G/P|< 
\infty$ for every $P={\bf P}(F)$ and ${\bf P}$ a parabolic of ${\bf G}$ defined over $F$.
\end{defn}
\begin{thm}[{Implications of \cite[Theorems 4.3 and 6.1]{AS}}] 
\label{thm:locally free sheaf}
Assume $(G,H)$ is an $F$-spherical pair and that $d_{H,\chi}(\rho) < \infty $.
\begin{enumerate}
\item  $\Hom_{G}(V,\rho \otimes \psi) \simeq \Hom_{\mathcal{O}_{\frak{X}_G}}(H_{\rho,V},\delta_{\psi})$ for every $V \in \repg$, where $\delta_{\psi}$ is the skyscraper sheaf over the trivial character $\psi \in \frak{X}_G$ with ring  $\mathcal{O}_{\frak{X}_G,\psi}/\frak{m}_\psi$, the residue field at $\psi$.
\item The module $H_{\rho,\ind_H^G  \chi^{-1} \delta}|_B$ is locally free over a smooth subvariety of $\frak{X}_G$.
\end{enumerate}

\end{thm}
\subsection{Direct integral decompositions of unitary representations}
\label{sec:Direct integral decompositions of unitary representations}
We follow \cite[Section 8]{KS18} and  \cite[Section 3.4]{Fuhr}, see these sources for a thorough treatment of direct integrals and disintegration of unitary representations.

Let $Z$ be a second countable topological space possessing a $\sigma$-finite Borel measure $\nu$.
%Given a collection of complex Hilbert spaces $(V_z,\angles{-,-}_z)_{z \in Z}$, we can form the space of sections 
%\[\prod\limits_{z\in Z} V_z=\curly{s: Z \to \coprod\limits_{z \in Z}V_z: s(z) \in V_z}.\]
%\begin{defn}
%We say a subspace $\mathcal{F} \subset \prod\limits_{z\in Z} V_z$ is a {\it measurable family of Hilbert spaces over $Z$} 
% if the following properties are satisfied:
% \begin{enumerate}
% \item For every $s, t \in \mathcal{F}$ the function $f_{s,t}: Z \to \complex$ by $f_{s,t}(z)=\angles{s(z),t(z)}_z$ is measurable.
% \item If $t$ is a section such that $f_{t,s}$ is measurable for every $ s \in \mathcal{F}$ then $t \in \mathcal{F}$.
%\item There exists a countable set $\curly{s_n}_{n \in \nats} \subset \mathcal{F}$ such that for every $z \in Z$ the collection $\curly{s_n(z) : n \in \nats}$ spans  a dense subset of $V_z$.
% \end{enumerate}
% \end{defn}
%\begin{defn}
%Let $\mathcal{F}$ be a measurable family of Hilbert spaces.  
%\begin{enumerate}
%\item We say that $s \in \mathcal{F}$ is {\it square-integrable}
%if $\int_Z \angles{s(z),s(z)}_z d\nu(z) < \infty$.
%\item We denote the space of all square-integrable sections (up to equality almost everywhere) by
%\[\int_Z^\oplus V_z d\nu(z).\]
%This space is called the {\it direct integral} of the measurable family $(V_z)_{z\in Z}$ with respect to $\nu$.
%It is a separable Hilbert space.
%\end{enumerate}
%\end{defn}

\begin{defn} Let $\alpha : S \to V= \int^\oplus_{Z} V_z d\nu(z) $ be a continuous morphism from a topological vector space $S$ to a Hilbert space $V$ given as a direct integral of the family $(V_z)_{z \in Z}$.
We say that $\alpha: S \to V$ is pointwise defined if there exists a family of morphisms $\curly{\alpha_z:S \to V_z}_{z \in Z}$ such that $\alpha(s)(z)= \alpha_z(s) $  $\nu$-almost everywhere 
for every $s \in S$.
\end{defn}
Assume $G$ is a second countable group of type I (see \cite[Definition 3.19]{Fuhr}).
The importance of this condition is that it implies 
the unitary representations of $G$ can be decomposed uniquely %(in the sense of \cite[Theorem 3.16(b)]{Fuhr}) 
as a direct integral of representations over the unitary dual $\widehat G$ of $G$ (see \cite[Theorem 3.16]{Fuhr} or \cite[Section 8.2]{KS18}).
By a theorem of Bernstein, every reductive group over a non-discrete, non-Archimedean local field is of type I (\cite{Ber3}).

Given a unitary representation $\rho$ of $G$, it possesses a {\it central decomposition} as follows:
\begin{thm}[{\cite[Theorem 3.24]{Fuhr}}]
\label{thm:central decomposition}
Let $G$ be a group of type I and let $(V,\rho)$ be a unitary representation of $G$. 
There exist a standard Borel measure $\nu_\rho$ on $\widehat{G}$ and a multiplicity function $m_\rho : \widehat{G} \to \nats_0 \cup \curly{\infty}$ such that
\[\rho\simeq \int^{\oplus}_{\widehat G} m_{\rho}(\pi) \pi d\nu_\rho(\pi).\]
\end{thm}
Let $\chi$ be a unitary character, and let $\nu_{W(\chi)}$ be the measure appearing in the central decomposition of $W(\chi):=L^2(H \backslash G,\chi^{-1}\delta)$ as in Theorem \ref{thm:central decomposition}. We call  $\nu_{W(\chi)}$ {\it the Plancherel measure} of $L^2(H \backslash G,\chi^{-1}\delta)$, and say that a representation $\rho \in \Irr{G}$ is {\it $(H\backslash G,\chi^{-1})$-tempered} 
if $\rho \in \supp(\nu_{W(\chi)})$. 
This is equivalent to 
 the existence of a functional $\xi \in (\rho^*)^{(H,\chi^{-1})}$ 
such that the matrix coefficient $m_{\xi,v}$ has certain good growth conditions for every $v \in \rho$ (see \cite[Page 666, Condition (***)]{Ber2}).
We finish this section with the next proposition, which is used to deduce the main result of Section \ref{sec:conversedirection}.
\begin{prop} 
[{See \cite[Section 2.3]{Ber2} for the unimodular case or \cite[Sections 3.2, 3.4 and 3.7]{Ber2} for the general one}]
\label{prop:Bernstein separation}
The natural embedding $\alpha: \ind_H^G \chi^{-1}\delta \to L^2(H \backslash G,\chi^{-1}\delta)$  is pointwise defined.
\end{prop}
%\begin{proof}
%See \cite[Section 2.3]{Ber2} for the unimodular case or \cite[Sections 3.2, 3.4 and 3.7]{Ber2} for the general one.
%\end{proof}
\section{Definition of $\mathcal{H}_\chi$ and Gelfand-Kazhdan conditions imply $\mathcal{H}_\chi$ is commutative}
Let $H \leq G$ be l.c.t.d groups as in Section \ref{sec:conventions}.
Since $G$ is unimodular, $\delta_G=1$, and we have $\delta^2=\delta_{H \backslash G} = \delta_H^{-1}$.
\begin{defn}\label{def:Hecke}
We define the Hecke algebra to be $\mathcal H_\chi:=\mathrm{End}_G(\ind_H^G\chi^{-1} \delta)$.
\end{defn}

\begin{rem}
If $G$ and $H$ are finite then $\mathcal H_\chi \simeq \mathbb C [ G ]^{(H,\chi)\times(H,\chi^{-1})}$,
so commutativity of $\H_\chi$ is equivalent to $(G,H)$ being a $\chi$-Gelfand pair (Definition \ref{def:TwistedGelfand}).
This shows $\H_\chi$ is a good candidate to generalize the usual Hecke algebra $\complex[H \backslash G /H]$.
In the next section we will see this also holds in the case where $H$ is compact and $G$ arbitrary.
\end{rem}
{While the Gelfand-Kazhdan criterion is usually phrased in terms of generalized functions, for us it will be easier to use distributions.
Since $G$ is unimodular, 
choosing a Haar measure gives a (non-canonical) identification between these spaces.
\begin{lem}
$C^{-\infty}(G) \simeq \Dist(G)$.
\end{lem}

We  wish to give $\H_\chi$ a geometric description as a space of invariant distributions.
Recall ($\ind_H^G \chi)^* \simeq \Dist(G)^{(H,\chi \delta_H)}$ by Lemma \ref{lem:distquotisom}.
This gives rise to an isomorphism
\[ \left(\left(\ind_H^G( \chi^{-1}\delta) \otimes \ind_H^G( \chi \delta)\right)^*\right)^{\Delta G} \simeq 
\Dist(G \times G)^{(H,\chi^{-1} \delta^{-1}) \times (H, \chi \delta^{-1}) \times \Delta G}.\]
 Consider the map 
$i_\chi : \mathcal H_\chi  \to \left(\left(\ind_H^G( \chi^{-1}\delta) \otimes \ind_H^G( \chi \delta)\right)^*\right)^{\Delta G}$ 
 defined by
\[ \langle i_{\chi}(\tau),f_1 \otimes f_2 \rangle=\int\limits_{H\backslash G} \tau(f_1(g))f_2(g) d\mu_{H \backslash G}.\]
Note that since $f_1\otimes f_2 \in  \ind_H^G( \chi^{-1}\delta) \otimes \ind_H^G( \chi \delta)$, we get $f_1\cdot f_2 \in \ind_H^G \delta_{H\backslash G}$, on which we have a $G$-invariant functional (\cite[Theorem 1.21]{BZ}).
Furthermore, $\mathcal{H}_\chi$ is indeed embedded in the space above:

\begin{lem} \label{lem:HeckeInj}
The map $i_\chi$ is an injection.
\end{lem}

\begin{proof}
Let $\tau \in \H_\chi$ be a non-zero element, and let $f_1 \in \ind_H^G \chi^{-1} \delta$ be a function such that 
$\tau(f_1) \neq 0$.
Take $e \neq g_0 \in G$ such that $\tau(f_1)(g_0) \neq 0$, define $f_2(g_0)=\overline{\tau(f_1)(g_0)}$ and extend it via $f_2(hg_0)=\chi(h)\delta(h)f_2(g_0)$ for every $h \in H$.
Now, take an open compact subgroup $K \subset G$ small enough such that $g_0 K  \cap H = \varnothing$ and such that $f_1(g_0K)=f_1(g_0)$, and set 
$f_2(g_0K)=f_2(g_0)$ and $f_2(x)=0$ if $x \notin Hg_0K$ (if $H$ is open take $K=\curly{e}$).
%We can choose such $K$ since $H$ is closed and $g_0 \notin H$.
Clearly, $f_2 \in \ind_H^G \chi \delta$, and since $(\tau(f_1)f_2)|_{Hg_0K}$ is positive and vanishes outside $Hg_0K$ we are done.
\end{proof}

\begin{lem}\label{lem:IsomSpaces}
We have,
\begin{equation*}
 \left(\left(\ind_H^G( \chi^{-1}\delta) \otimes \ind_H^G( \chi \delta )\right)^*\right)^{\Delta G}
\simeq \Dist(G)^{(H,\chi^{-1} \delta^{-1}) \times (H, \chi \delta^{-1}) }. 
\tag{$*$}
\end{equation*}
\end{lem}
\begin{proof}
To prove the lemma we apply Lemma \ref{lem:distquotisom} twice to show that both spaces are isomorphic to 
\[\Dist(G\times G)^{(H,\chi^{-1} \delta^{-1}) \times (H, \chi \delta^{-1} ) \times \Delta G}.\]
\end{proof}

Assume the Gelfand-Kazhdan conditions  with respect to $\chi\delta$ hold (see Theorem \ref{prop:GK1}).
This means there exists an anti-involution $\sigma$ of $G$ preserving $H$ such that $\xi^\sigma=\xi$ for all distributions $\xi$ belonging to the right hand side of $(*)$. 
Further assume that $\chi$ satisfies $\chi^\sigma=\chi$, and recall $\delta^\mu = \delta$ by Lemma \ref{lem:MeasureUnderInv}, where 
$\mu(x)=\sigma(x^{-1})$. 

By the isomorphism above, $\sigma$ can be translated  to an involution $\theta$ on the left hand side of $(*)$.

\begin{prop} \label{prop:ThetaAction}
Let $\varphi \in \left(\left(\ind_H^G( \chi^{-1}\delta) \otimes \ind_H^G( \chi \delta )\right)^*\right)^{\Delta G} $. The involution $\theta$ induced by $\sigma$ acts as follows:
\[\langle\varphi^\theta,f_1 \otimes f_2  \rangle=\langle\varphi,f_2^{\mu} \otimes f_1^{ \mu} \rangle.\]
\end{prop}
We divide the proof into two lemmas, first passing through the involution $\tilde{\theta}$ induced on the space $\Dist(G \times G)^{  \Delta G}$, and then showing that $\tilde{\theta}$ gives rise to the involution $\theta$ as above.
\begin{lem} \label{lem:inducedinvolutiononGxG}
The anti-involution $\sigma$ on $\Dist(G)$ induces the following on $\Dist(G \times G)^{  \Delta G}$:
\[\langle \varphi^{\tilde \theta},f_1 \otimes f_2\rangle=\langle \varphi ,f_2^{\mu} \otimes f_1^{\mu} \rangle.\]
\end{lem}
\begin{proof} 
Consider the isomorphism 
 $\Psi: \Dist(G) 
 \simeq (\ind_{\Delta G}^{G \times G} 1)^* 
 \xrightarrow{\sim} \Dist(G \times G)^{\Delta G}$ as in Lemma \ref{lem:distquotisom}, 
\[\angles{\Psi(\eta),f}=\angles{\eta,  \Phi_{f}},\] 
where $\Phi_f(q)=\int\limits_{ G} f(qr,r)dr$  (here the $G\times G$-action on $\ind_{\Delta G}^{G \times G} 1$  is on the left).
We have,
\begin{gather*}
\angles{\Psi(\eta^\sigma),f_1 \otimes f_2}=\angles{\eta,\Phi_{f_1 \otimes f_2}^\sigma}, \\
\angles{\Psi(\eta)^{\tilde \theta},f_1 \otimes f_2}=\angles{\Psi(\eta),f_2^\mu \otimes f_1^\mu} =\angles{\eta,\Phi_{f_2^\mu \otimes f_1^\mu}}.
\end{gather*}
Thus it is enough to show that $\Phi_{f_2^\mu \otimes f_1^\mu}=\Phi_{f_1 \otimes f_2}^\sigma$. We calculate and get the following:
\begin{flalign*}
\Phi_{f_2^\mu \otimes f_1^\mu}(q)
&=\int\limits_{ G} f_2^\mu \otimes f_1^\mu (qr,r)dr
=\int\limits_{ G} f_1 \otimes f_2 (\mu(r),\mu(qr))dr.
\end{flalign*}
Substitute $r'= \mu(r)=\sigma(r^{-1})$. 
Since $\mu$ is an involution, by Lemma \ref{lem:MeasureUnderInv} we have $\mu_* dr = dr'$. 
We are now finished by the following:
\begin{flalign*}
&\int\limits_{ G} f_1 \otimes f_2 (r',\mu(q)r')dr'=\int\limits_{ G} f_1 \otimes f_2 (\sigma(q)r',r')dr'
=\Phi_{f_1 \otimes f_2}^\sigma(q).
\end{flalign*}
\end{proof}

\begin{lem}
The involution $\tilde \theta$ on $\Dist(G\times G)^{(H,\chi^{-1} \delta^{-1}) \times (H, \chi \delta^{-1}) \times \Delta G}$ induces the involution $\theta$ on $\left(\left(\ind_H^G( \chi^{-1}\delta ) \otimes \ind_H^G( \chi \delta)\right)^*\right)^{\Delta G}$.
\end{lem}
\begin{proof}
As before have a map, 
$\Psi: ((\ind_H^G( \chi^{-1}\delta) \otimes \ind_H^G( \chi \delta))^*)^{\Delta G} 
\to \Dist(G\times G)^{(H,\chi^{-1} \delta^{-1}) \times (H, \chi \delta^{-1}) \times \Delta G}$
by 
\[\angles{\Psi(\xi),f_1 \otimes f_2} = \angles{\xi, \Phi_{f_1, \chi^{-1} \delta} \otimes \Phi_{f_2, \chi \delta}}\]
where $\Phi_{f, \psi}(g)=\int\limits_H f(hg)\psi^{-1}(h)dh $ as in Lemma \ref{lem:distquotisom} and $\psi\in \curly{\chi^{-1}\delta,\chi\delta}$.
Now, expanding,
\begin{flalign*}
\langle \Psi(\xi)^{ \tilde \theta},f_1 \otimes f_2\rangle &
= \langle \Psi(\xi),f_2^{\mu} \otimes f_1^{\mu}\rangle 
= \angles{\xi, \Phi_{f_2^\mu, \chi^{-1} \delta} \otimes \Phi_{f_1^\mu, \chi \delta}}.
\end{flalign*}
For every character $\psi$ of $H$ and $f \in C_c^\infty(G)$, it holds that,
\[
\Phi_{f^\mu,\psi}(x)
=\int\limits_H f(\mu(h)\mu(x))\psi^{-1}(h)dh
=\int\limits_H f(h\mu(x))\psi^{-1}(\mu(h))dh
=\Phi^\mu_{f,\psi^\mu},
\]
where $\mu_*(dh)=dh$ by Lemma \ref{lem:MeasureUnderInv}.
Back to our calculation, using the relations $\delta^\mu=\delta$ and $\chi^\sigma =\chi$
we conclude the desired statement:
\begin{flalign*}
\angles{\xi, \Phi_{f_2^\mu, \chi^{-1} \delta} \otimes \Phi_{f_1^\mu, \chi \delta}}
&=\angles{\xi, \Phi^\mu_{f_2, (\chi^{-1} \delta)^\mu} \otimes \Phi^\mu_{f_1, (\chi \delta)^\mu}}
=\angles{\xi^\theta, \Phi_{f_1, \chi^{-1} \delta} \otimes \Phi_{f_2, \chi \delta}} 
=\angles{\Psi(\xi^\theta), f_1 \otimes f_2}.
\end{flalign*}
\end{proof}
Using Proposition \ref{prop:ThetaAction} we can now show the commutativity of $\H_\chi$.
Recall we assume $\chi^\sigma = \chi$.
\begin{thm} \label{thm:HeckeCommutative}
Assume the conditions of the Gelfand-Kazhdan criterion hold with $\psi=\chi^{-1}\delta$ (Theorem \ref{prop:GK1}), i.e. there exists an anti-involution $\sigma$ of $G$
such that $\sigma(H)=H$ and  $\xi^\sigma= \xi $ for every $\xi \in \Dist(G)^{(H,\chi^{-1}\delta^{-1}) \times (H,\chi^\sigma\delta^{-1})}$.
Then $\H_\chi$ is commutative.
\end{thm}
\begin{proof}
Take $\tau_1,\tau_2 \in \mathcal H_\chi=\mathrm{End}_G(\ind_H^G \chi^{-1}\delta)$. 
Since $i_\chi$ is injective, 
 it is enough to show that $i_\chi(\tau_1 \circ \tau_2) = i_\chi(\tau_2 \circ \tau_1)$, and since every $\xi$ in $\Dist(G)^{(H,\chi^{-1}\delta^{-1}) \times (H,\chi^\sigma\delta^{-1})}$ is fixed under $\sigma$,
for every $\tau \in\H_\chi$ we have $i_\chi(\tau)^\theta=i_\chi(\tau)$ where $\theta$ is as in Proposition \ref{prop:ThetaAction}.
The statement is now reduced to the following calculation:
\begin{flalign*}
\langle i_\chi (\tau_1 \circ \tau_2),f_1 \otimes f_2\rangle &
= \langle i_\chi (\tau_1), \tau_2(f_1) \otimes f_2\rangle
 =\langle i_\chi (\tau_1)^\theta , \tau_2(f_1) \otimes f_2\rangle \\
&= \langle i_\chi (\tau_1) ,  f_2^\mu \otimes \tau_2(f_1)^\mu \rangle 
=\langle i_\chi (\mathrm{Id}) ,  \tau_1(f_2^\mu) \otimes \tau_2(f_1)^\mu \rangle 
\\
& = \langle i_\chi (\mathrm{Id}) , \tau_2(f_1) \otimes  {\tau_1(f_2^\mu)}^\mu \rangle 
= \langle i_\chi (\tau_2) ,f_1 \otimes  {\tau_1(f_2^\mu)}^\mu \rangle\\
& 
=\langle i_\chi (\tau_2) , {\tau_1(f_2^\mu)} \otimes f_1^\mu\rangle 
 =\langle i_\chi (\tau_2 \circ \tau_1) , f_2^\mu \otimes f_1^\mu\rangle  \\
 &
= \langle i_\chi (\tau_2 \circ \tau_1) , f_1 \otimes f_2 \rangle.
 \end{flalign*}
\end{proof}
\section{$\mathcal{H}_\chi$ commutative implies $(G,H)$ is a $\chi$-Gelfand pair - geometric cases}
Let $\mathcal H_\chi:=\End_G(\ind_H^G \chi^{-1} \delta)$ be as in Definition \ref{def:Hecke} and 
recall $G$ is an l.c.t.d unimodular group with a closed, not necessarily unimodular subgroup $H$.
In this section we prove Theorem \ref{thm:geometric cases}.
We start by showing that $\H_\chi$ generalizes the Hecke algebra from the classical case, that is for compact $H$ the pair $(G,H)$ is a $\chi$-Gelfand pair $\iff$ $\H_\chi$ is commutative.

Let $A(G)$ be the algebra of smooth, compactly supported measures on $G$, and recall that for compact $K$ the 
pair $(G,K)$ is a $\chi$-Gelfand pair if and only if $A(G)^{(K,\chi) \times (K,\chi^{-1})}$ is commutative (see Section \ref{subsub:CompactSubgroup} for $\chi=1$, general $\chi$ is similar).

\begin{prop}
Assume $K:=H$ is compact.
%, 
%and let $A(G)$ be the algebra of smooth, compactly supported measures on $G$.
We have $A(G)^{(K,\chi) \times (K,\chi^{-1})} \hookrightarrow \H_\chi$ as algebras.
If furthermore $K$ is open, then this map is an isomorphism.
\end{prop}
\begin{proof}
Let $\mu_G$ be a choice of a Haar measure on $G$. 
We have a map $\psi:A(G)^{(K,\chi) \times (K,\chi^{-1})} \to \H_\chi$, 
\[
\psi(f' \mu_G)(f)(g)=(  f' \mu_G*f)(g)(x)= \int_G f(g)f'(xg^{-1}) d\mu_G.
\]
It is a well defined $G$-morphism since $\psi(f' \mu_G)(f) \in \ind_K^G \chi^{-1}$:
\[
\psi(f' \mu_G)(f)(hxg')
=\int_G f(g)f'(hxg'g^{-1}) d\mu_G
=\chi^{-1}(h)\psi(f' \mu_G)(R_{g'}f)(x).
\]
It is injective since $\psi$ commutes with embedding of both spaces into $\Dist(G)^{(K,\chi) \times (K,\chi^{-1})}$, which for $\H_\chi$ is given by Frobenius reciprocity. 
If $K$ is open, Lemma \ref{lem:standardfacts}(4) shows $\psi$ is an isomorphism.
\end{proof}

\begin{cor} \label{cor: GP1 for compact subgroup}
Let $K:=H$ be a compact subgroup, and assume $G$ is second countable. %of Type I 
Then $\H_\chi$ is commutative $\iff$ $(G,K)$ is a $\chi$-Gelfand pair.
\end{cor}
\begin{proof}
If $\H_\chi$ is commutative, then by the previous lemma so is $A(G)^{(K,\chi) \times (K,\chi^{-1})}$,
implying that $(G,K)$ is a $\chi$-Gelfand pair.

%The converse follows 
%by Theorem \ref{thm:GPimpliesCom} (any smooth character of a compact group is unitary).
The converse follows since by \cite[Separation Lemma]{BR}, $\H_\chi$ embeds in $\prod\limits_{\rho \in \Irr{G}} \End_\complex(\Hom_G(\ind_H^G \chi^{-1}),\rho)$ which is commutative if $(G,H)$ is a Gelfand pair (see Theorem \ref{thm:GPimpliesCom} for a proof assuming $G$ is of type I).
\end{proof}
\begin{prop} \label{GP1 for G/H compact}
Let $\chi$ be a unitary character of $H$, and assume
that $H\backslash G$ is compact.
Then $\mathcal H_\chi$ is commutative $\iff$ $(G,H)$ is a $\chi$-Gelfand.
\end{prop}
\begin{proof}
Since $\chi$ is unitary, $\ind_H^G \chi^{-1} \delta$ is unitarizable, and
by compactness of $H \backslash G$ the representation $\Ind_H^G  \chi^{-1}\delta=\ind_H^G  \chi^{-1}\delta$ is admissible.
This implies it is semi-simple.
Consequently, by Schur's lemma $\H_\chi = \End_G(\ind_H^G \chi^{-1} \delta) $ decomposes as a product of matrix algebras, 
$\prod\limits_{\alpha \in I} \mathrm{M}_{n_\alpha}(\complex)$, where 
\[n_\alpha= \dim_\complex \Hom_G(\ind_H^G \chi^{-1} \delta,\rho_\alpha),
\]
and $\curly{\rho_\alpha}_{\alpha \in I}$ are the smooth irreducible representations of $G$ (up to equivalence).
In particular $\H_\chi$ is commutative if and only if $n_\alpha \leq 1 $ for every $\alpha \in I$.
\end{proof}
\subsection{Hecke pairs and Schlichting completions} 
\label{sec:Hecke pairs}
We now move to consider the case of {\it Hecke pairs}. These are pairs $(G,H)$ which behave as if $H$ was compact.
\begin{defn} Let $G$ be a l.c.t.d group with a closed subgroup $H$.
\begin{enumerate}
\item {We say $H$ is commensurated in $G$ if $|HgH/H|< \infty$ for every $g \in G$.}
\item We say that $(G,H)$ is a {\it Hecke pair} if $H$ is open and commensurated in $G$.
\end{enumerate}
\end{defn}

\begin{rem}
While Hecke pairs are usually defined in the context of discrete groups, there exist interesting non-discrete 
Hecke pairs $(G,H)$ where $H$ is not compact, cocompact or normal (see the pair in Corollary \ref{cor:BaderGroupII}).
%(e.g. $(G(F),G(F)_v)$ in \cite{LB16}).
\end{rem}%%
Let $H \leq G$ be l.c.t.d groups as in Section \ref{sec:conventions}.
If $(G,H)$ is a Hecke pair, one can complete the group $G$ with respect to a certain topology arising naturally from $H$ and
get a locally compact, totally disconnected, Hausdorff topological group $\hat{G}$ and a homomorphism $\beta:G \to \hat{G}$ such that $\hat{H}:=\overline{\beta(H)}$ is open and compact. 
The group $\hat{G}$ is usually called in the literature  
{\it the Schlichting completion} of $(G,H)$ 
(see \cite[Definitions 3.3, 3.4 and 3.7]{KLQ08} or \cite[Definition 5.2]{RW}).

More explicitly, for a set $S \subset G / H$  define $H_S=H \cap \bigcap\limits_{g \in S} g H g^{-1}$ and let
$\mathcal{S}:=\{S \subset G/H: S  \text{ is finite}\}$.
If $(G,H)$ is a Hecke pair and $S \in \mathcal{S}$, then $H_S$ has finite index in $H$.
The collection $\{H_S : S \in \mathcal{S}\}$ forms a {basis} at identity for a group topology on $G$,
and completing $G$ with respect to it yields the Schlichting completion of $(G,H)$.

Alternatively, note $\{G/H_S : S\in \mathcal{S}\}$ forms an inverse system. The Schlichting completion of $(G,H)$ can be described as follows:
\begin{lem}[{\cite[Proposition 3.10]{KLQ08}}]
\label{lem: inverselimit}
$\hat {G}\simeq  \lim\limits_{\xleftarrow[{S \in \mathcal{S}}]{} } G/H_S$.
\end{lem}
\begin{rem}
Since we demand $\hat{G}$ is Hausdorff, we have $ \ker \beta =\bigcap\limits_{g \in G} g H g^{-1}$.
\end{rem}
Let $\mathrm{Rep}(G)$ be the category of smooth representations of a group $G$, and for $H \leq G$ set 
\[\mathrm{Rep}_H(G):=\curly{V \in \repg : V= \bigcup\limits_{S \in \mathcal{S}}V^{H_S} }.\]
 Note that an irreducible representation $\rho \in \mathrm{Rep}(G)$ lies in $\mathrm{Rep}_H({G})$ if $\rho^{H_S}\neq \{0\}$ for some $S \in \mathcal{S}$.
\begin{thm} \label{thm:equivofcats}
Let $(G,H)$ be a Hecke pair. Then $\mathrm{Rep}_H(G) \cong\mathrm{Rep}(\hat{G}) $.
\end{thm}
\begin{rem}
Note that every $\rho \in \mathrm{Rep}_H(G)$ is trivial on the kernel of $\beta$ since $\ker (\beta) \subset H_S$ for every $S \in \mathcal{S}$. It follows that $\mathrm{Rep}_H(G) \cong \mathrm{Rep}_{H/N}(G/N)$ for $N=\ker(\beta)$, so it is enough to prove the theorem in the case where $\beta$ is injective.
\end{rem}

\begin{proof}[ {Proof of Theorem \ref{thm:equivofcats}}]
For $(V,\rho)\in \mathrm{Rep}_H(G)$ and $S \in \mathcal{S}$, 
set 
$\rho_S : G/H_S \to \mathrm{Hom}_\complex(V^{H_S},V)$ 
by $\rho_S(gH_S)(v)=\rho(g)(v)$.
This is well defined since every $h_S \in H_S$ acts trivially on $V^{H_S}$.
Now, by  Lemma \ref{lem: inverselimit} we have maps $\hat{G} \to G/H_S$, and by precomposing we get 
maps $\hat{\rho}_S : \hat{G} \to\mathrm{Hom}_\complex(V^{H_S},V)$.

{Note that $\curly{\mathrm{Hom}_\complex(V^{H_S},V)}_{S \in \mathcal{S}}$ forms an inverse system with respect to the natural restriction maps.
These maps commute with the maps $\curly{\hat{\rho}_S}_{S \in \mathcal{S}}$ defined above, and furthermore
\[
\lim\limits_{\xleftarrow[{S \in \mathcal{S}}]{} }\mathrm{Hom}_\complex(V^{H_S},V)
=\mathrm{Hom}_\complex(\lim\limits_{\xrightarrow[{S \in \mathcal{S}}]{} } V^{H_S},V) 
= \mathrm{Hom}_\complex(V,V).
\]
By the universal property of 
$\lim\limits_{\xleftarrow[{S \in \mathcal{S}}]{} }\mathrm{Hom}_\complex(V^{H_S},V)$ we get 
a unique map $F(\rho):=\hat{\rho}:\hat{G} \to \mathrm{End}_\complex(V)$.
By the construction $\hat{\rho}|_G = \rho$; if $v\in V$ then $v \in V^{H_S}$ for some $S\in \mathcal{S}$, we get for every $\beta(g)\in \hat{G}$,
\[
\hat{\rho}(\beta(g))(v)=\hat{\rho}_S(\beta(g))(v)=\rho_S(gH_S)(v)=\rho(g)(v).
\]
Note that $\hat{\rho}$ is continuous, 
i.e. for a given $v \in V$ the map $\hat{\rho}_v:g  \mapsto \hat{\rho}(g)(v) $ is continuous. 
This implies that $\hat{\rho}$ is
smooth, since for every $v \in V$ we have $\beta(H_S) \subset \mathrm{Stab}_{\hat{G}}(v)$ for some $S \in \mathcal{S}$,
so $\overline{\beta(H_S)}$ stabilizes $v$.
In particular, $\hat{\rho}$ is a homomorphism of groups to $\mathrm{GL}(V)$,
since it is a homomorphism on a dense set $G \subset \hat{G}$, and similarly it follows that $F$ carries morphisms to morphisms.
We conclude that $F((V,\rho))=(V,\hat{\rho}) \in \mathrm{Rep}(\hat{G})$.}

Conversely, define $\mathrm{res}^{\hat{G}}_{G}: \mathrm{Rep}(\hat{G}) \to \repg$ by $\mathrm{res}^{\hat{G}}_{G}(\hat{\rho})(g) =\hat{\rho}(\beta(g))$, acting on the same space. 
Since every $\hat{V}$  is smooth and $\beta$ is continuous, every $\mathrm{res}^{\hat{G}}_{G}(\hat{V})$ is smooth.
Furthermore, since a basis for the topology of $\hat{G}$ at identity is given by $\curly{\overline{\beta{(H_S)}}: S \in \mathcal{S}}$, every $v \in \hat{V}$ is stabilized by some $\overline{\beta({H_S})}$, and thus every $v$ is stabilized by some $H_S$ as a representation of $G$.

Since by the construction $\mathrm{res}^{\hat{G}}_{G} \circ F =\mathrm{Id}_{\mathrm{Rep}_H(G)}$ and $F \circ \mathrm{res}^{\hat{G}}_{G} = \mathrm{Id}_{\mathrm{Rep}(\hat{G})}$, we are done.
\end{proof}
\begin{lem} \label{lem:isomorphism of Hecke algs}
Let $(G,H)$ be a Hecke pair and let $\chi : H \to \complex^{\times} $ be a character such that $\chi(H_S)=1$ for some $S \in \mathcal{S}$. 
Then $\H_\chi(G,H) \simeq \H_{\hat \chi}(\hat{G},\hat{H})$, where $\hat{\chi}{|_H}=\chi$.
\end{lem}
\begin{proof}
By the equivalence of categories, $\End_G(\ind_H^G \chi^{-1}) \simeq \End_{\hat G}(F(\ind_H^G \chi^{-1}))$.
Now $F(\ind_H^G \chi^{-1})=\ind_{\hat{H}}^{\hat{G}} \hat{\chi}^{-1}$, where $\hat{\chi}: \hat{H} \to \complex^{\times}$
is a character of $\hat{H}$ extending $\chi$. Note $H$ is open so $\delta=1$, and that an extension $\hat\chi$ exists since $H/ \ker \chi \simeq \hat{H}/ \ker \chi$ as finite groups.
\end{proof}
\begin{prop} \label{cor:HeckeCase}
Let $(G,H)$ be a Hecke pair and $\chi : H \to \complex^{\times}$ as in Lemma \ref{lem:isomorphism of Hecke algs}. 
Then $\H_\chi$ is commutative $\iff$ $(G,H)$ satisfies the $\chi$-Gelfand property with respect to every irreducible representation in $\mathrm{Rep}_H(G)$.
In particular, if $\H_\chi$ is commutative 
then 
%$(G,H)$ satisfies the $\chi$-Gelfand property with respect to every smooth irreducible admissible representation $\rho$ of $G$.
$\dim_\complex \Hom_H(\rho|_H,\chi)\leq1$ for every smooth irreducible admissible representation $\rho$ of $G$.
%$(G,H)$ satisfies the $\chi$-Gelfand property with respect to every irreducible admissible representation of $G$.
%
\end{prop}
\begin{proof}%[ {Proof of Corollary \ref{cor:HeckeCase}}]
Let $M=\bigcup\limits_{S \in \mathcal{S}}(\Ind_H^G \chi)^{H_S}$, then $F(M)=\Ind_{\hat{H}}^{\hat{G}} \hat{\chi}$ where $\hat{\chi}$
is a character of $\hat{H}$ extending $\chi$.
Let $\rho \in \mathrm{Rep}_H(G)$ be an irreducible representation.
By the equivalence of categories, 
 we have 
\[
\dim_\complex\Hom_G(\rho,\Ind_H^G \chi)
=\dim_\complex\Hom_G(\rho,M)
=\dim_\complex\Hom_{\hat{G}}(F(\rho),\Ind_{\hat{H}}^{\hat{G}} \hat{\chi}),
\]
and in particular $(G,H)$ is a $\chi$-Gelfand pair with respect to $ \mathrm{Rep}_H(G)$ $\iff$ $(\hat{G},\hat{H})$ is a $\hat{\chi}$-Gelfand pair.
Using Lemma \ref{lem:isomorphism of Hecke algs}, we have $\H_\chi(G,H)  \simeq \H_{\hat{\chi}}(\hat{G},\hat{H})$,
and since $\hat{H}$ is compact and open in $\hat{G}$, we get that $(\hat{G},\hat{H})$ is a Gelfand pair $\iff$ 
$\H_{\hat{\chi}}(\hat{G},\hat{H}) \simeq \H_\chi(G,H)$ is commutative. 

If $\rho$ is an irreducible admissible representation of $G$ and $\dim(\rho^*)^H>0$, then $\tilde{\rho}$ is irreducible. Since $H$ is open $\dim (\rho^*)^H=\dim (\tilde{\rho})^H>0$ and we have $\tilde{\rho}\in \mathrm{Rep}_H(G)$. The claim now follows by the first part of the proof.
\end{proof}

\begin{cor}
\label{cor:BaderGroupII}
Let $\mathcal{T}_d$ be a $d\geq 3$ regular tree, let 
$G \leq \mathrm{Aut}(\mathcal{T}_d)$ 
be the subgroup of automorphisms of $\mathcal{T}_d$ such that their local action is prescribed by a fixed finite permutation group $\{e\}\neq\Lambda < S_d$ around all but finitely many vertices, equipped with a l.c.t.d topology as in \cite{LB16}, and let $H$ be a stabilizer of a vertex $v$ in $\mathcal{T}_d$. 
Then $d_{H,1}(\rho) \leq 1$ for every $\rho \in \mathrm{Irr}(G)$ which is either admissible or $\rho \in \mathrm{Rep}_H(G)$.
\end{cor}
\begin{proof}
Since ${G}$ and $H$ are dense in $\mathrm{Aut}(\mathcal{T}_d)$ and $\mathrm{Stab}_{\mathrm{Aut}(\mathcal{T}_d)}(v)$ respectively, where $\mathrm{Stab}_{\mathrm{Aut}(\mathcal{T}_d)}(v)$ is open and compact, it follows that the Schlichting completion of $(G,H)$ is $(\mathrm{Aut}(\mathcal{T}_d),\mathrm{Stab}_{\mathrm{Aut}(\mathcal{T}_d)}(v))$. 
The claim now follows from Lemma \ref{lem:isomorphism of Hecke algs} and Proposition \ref{cor:HeckeCase}, since the pair $(\mathrm{Aut}(\mathcal{T}_d),\mathrm{Stab}_{\mathrm{Aut}(\mathcal{T}_d)}(v))$ is a Gelfand pair  (see \cite{Ol77}).
\end{proof}
\begin{rem} Let $(G,H)$ and $\chi$ be as in Lemma \ref{lem:isomorphism of Hecke algs} and let $\rho$ be an irreducible representation of $G$.
If $\rho^{(H,\chi^{-1})}\neq 0$ then $\rho \in\mathrm{Rep}_H(G)$. 
If $\rho$ is not admissible, it may happen
that $(\rho^*)^{(H,\chi^{-1})}\neq 0$ but $\rho \notin\mathrm{Rep}_H(G)$, and
in particular
even if $\H_\chi$ is commutative we might have $\dim_\complex(\rho^*)^H > 1$. 
\end{rem}
\begin{rem} If $H$ is open but $(G,H)$ is not a Hecke pair Proposition \ref{cor:HeckeCase} fails.
Let $T_3$ be the subgroup of diagonal matrices in $\mathrm{GL}_3(\Qp)$, let $G:=N_{\mathrm{GL}_3(\Qp)}(T_3)$ be its normalizer and set $H=\mathrm{diag}(\Zp^\times,\Qp^\times,\Qp^\times) \leq T_3$.
One can show $\mathcal{H} \simeq \complex[\ints \times \ints/ 2 \ints]$ is abelian, but the irreducible $G$-representation
$V=\{(x_1,x_2,x_3) \in \complex^3 : \sum x_i=0\}$ with the $G$-action given by permuting the coordinates
has $\dim V^H=2$. In particular $(G,H)$ is not a Gelfand pair.
\end{rem}
\subsubsection{Hecke pairs of algebraic groups}
An algebraic group ${\bf G}$ defined over $\rats$ naturally gives rise to a Hecke pair by considering $({\bf G}(\rats),{\bf G}(\ints))$ (see \cite[Section 4.1, Corollary 1]{PR94}).
{This has been famously used in \cite{BC95}.}
Furthermore, in this case the Schlichting completion can be determined explicitly. 
One can then translate information from
the relative representation theory of the completion to that of  $({\bf G}(\rats),{\bf G}(\ints))$.

Let $\mathbb{A}_f:=\sideset{}{'}\prod\limits_{p\text{ prime}} \Qp $ be the ring of finite adeles, 
and set $\hat{\ints}:=\prod\limits_{p\text{ prime}} \Zp$.
We will need the following.
\begin{thm}[Strong approximation, see {\cite[Theorem 7.12]{PR94}}]
\label{thm:strong approx}
Let ${\bf G}$ be a connected, simply connected reductive algebraic group over $\rats$ with no $\rats$-simple components ${\bf G}^i$ where ${\bf G}^i(\reals)$ is compact.
Then the diagonal embedding~ $\psi:({\bf G}(\rats),{\bf G}(\ints)) \hookrightarrow ({\bf G}(\mathbb{A}_f),{\bf G}(\hat{\ints}))$  is dense.
\end{thm}

Assume ${\bf G}$ satisfies the conditions of the strong approximation theorem.
Then $\psi$ is a completion map, and by the universal property of the Schlichting completion 
(\cite[Theorem 5.4]{RW}) the pair $({\bf G}(\mathbb{A}_f),{\bf G}(\hat{\ints}))$ is the Schlichting completion of 
$({\bf G}(\rats),{\bf G}(\ints))$ (up to compact center). By Lemma \ref{lem:isomorphism of Hecke algs}, this implies 
\[\H({\bf G}(\rats),{\bf G}(\ints)) \simeq \H({\bf G}(\mathbb{A}_f),{\bf G}(\hat{\ints}))
\simeq {\bigotimes}'\H({\bf G}(\rats_p),{\bf G}(\ints_p))= {\bigotimes}' C_c^\infty({\bf G}(\rats_p))^{{\bf G}(\ints_p) \times {\bf G}(\ints_p)}.
\] 
We can thus restate the classical result that $({\bf G}(\mathbb{A}_f),{\bf G}(\hat{\ints}))$ is a Gelfand pair in the language of discrete groups using Proposition \ref{cor:HeckeCase}. 
Note that in the case ${\bf G}= \mathrm{GL}_2$ the Hecke algebra
$\H({\bf G}(\rats),{\bf G}(\ints))$ recovers the classical algebra of Hecke operators on modular forms \cite{Co04}.
\begin{cor} \label{cor:discrete algebraic pairs}
Let ${\bf G}$ be a connected reductive algebraic group defined over $\rats$ satisfying the strong approximation property (Theorem \ref{thm:strong approx}).
%Let ${\bf G}$ be as above. 
Then the following pairs $(G,H)$ satisfy %$\rho \in \mathrm{Rep}_{{\bf G}(\ints)}({\bf G}(\rats))$,
$d_{H,1}(\rho)=\dim_\complex \Hom_{H}(\rho|_{H},\complex)\leq 1$ for every irreducible representation 
$\rho \in \mathrm{Rep}_{H}(G)$:
%\[d_{{\bf{G}}(\ints),1}(\rho)=\dim_\complex \Hom_{{\bf G}(\ints)}(\rho|_{{\bf G}(\ints)},\complex)\leq 1.\]
\begin{enumerate}
\item $(G,H)=({\bf G}(\rats),{\bf G}(\ints))$.
%\item $(G,H)=({\bf G}(K),{\bf G}(\mathcal{O}_K))$ where $K$ is a number field, $\mathcal{O}_K$ its ring of integers and ${\bf G}$ is suitable.
\item $(G,H)=({\bf G}(\ints[\frac{1}{p}]),{\bf G}(\ints))$. %(${\bf G}$ satisfies the desired approximation property, see \cite[Theorem 7.7]{PR94}).
\item $(G,H)=({\bf G}(\rats),{\bf G}(\ints[S_p]))$, where $p$ is prime, $S_p=\curly{\frac{1}{p'} :p' \neq p\text{ and }p'\text{ is prime}}$.
\end{enumerate}
\end{cor}
\begin{rem}
Item (1) of Corollary \ref{cor:discrete algebraic pairs} above can be adapted to the case of a number field $K$.
\end{rem}
%\begin{rem}
%Similarly, the following have the Gelfand property in the sense of Corollary \ref{cor:discrete algebraic pairs}.
%\begin{enumerate}
%\item $({\bf G}(K),{\bf G}(\mathcal{O}_K))$ where $K$ is a number field, $\mathcal{O}_K$ its ring of integers and ${\bf G}$ is suitable.
%\item $({\bf G}(\ints[\frac{1}{p}]),{\bf G}(\ints))$. %(${\bf G}$ satisfies the desired approximation property, see \cite[Theorem 7.7]{PR94}).
%\item $({\bf G}(\rats),{\bf G}(\ints[S_p]))$, where $p$ is prime, $S_p=\curly{\frac{1}{p'} :p' \neq p\text{ and }p'\text{ is prime}}$ and ${\bf G}$ as in (2).
%\end{enumerate}
%\end{rem}
\section{$\mathcal{H}_\chi$ is commutative implies $(G,H)$ is a cuspidal $\chi$-Gelfand pair}
\label{sec:conversedirection}
Let ${\bf G}$ be a connected reductive group with a 
{Zariski closed subgroup ${\bf H}\leq {\bf G}$ and set $G={\bf G}(F)$ and $H={\bf H}(F)$}.
In this section we show that for $F$-spherical pairs (Definition \ref{def:sphericalpair}) the cuspidal part of $\mathcal{H}_\chi$ is commutative if and only if $(G,H)$
is a cuspidal $\chi$-Gelfand pair, i.e. $\dim_\complex\Hom_H( \rho|_H,\chi\delta) \leq 1$ for every $\rho \in \mathrm{Irr}(G)$.

Since Theorem \ref{thm:BerDecomp} implies that $\H_\chi
= \End_G(\ind_H^G \chi^{-1})$ 
decomposes as a direct sum of endomorphism rings over the different blocks, 
if it is commutative then in particular so are its projections to each cuspidal block.
This allows us to study 
the relation between commutativity of $\H_\chi$ and the twisted Gelfand property of $(G,H)$
  on each  block separately.
 Since cuspidal blocks have a relatively simple description, they are easier to analyze than a general block.

Let $\rho$ be an irreducible cuspidal representation of $G$.
Recall that in Section \ref{sec:Bdecomposition} we defined,
\[
Z:=Z(R_\rho)\subseteq B:=\mathcal{O}(\frak{X}_G) \subseteq R_\rho:=\End_G(\ind_{G_0}^G (\rho|_{G_0})),
\] 
and that furthermore we had $H_{\rho,V}:=\Hom_G(\ind_{G_0}^G (\rho|_{G_0}),V)$ for a $G$-representation $V$.

\begin{prop} \label{prop:commifftensoriscomm}
Let $Z_\rho \subset B \subset R_\rho$ be as above. If $N$ is an $R_\rho$-module then 
$\End_{R_\rho} N$ is commutative $\iff$ $\End_{B\otimes_{Z_\rho} {R_\rho}} B\otimes_{Z_\rho} N$ is commutative.
\end{prop}
We prove Proposition \ref{prop:commifftensoriscomm} in two steps (Lemma \ref{lem:commifftensoriscomm part 1} and Proposition \ref{prop:commifftensoriscomm part 2} below).
\begin{lem} \label{lem:commifftensoriscomm part 1}
Let $Z \subseteq B$ be commutative rings such that $\mathrm{Spec} (B) \to \mathrm{Spec} (Z)$  is faithfully flat, and  let $R$ be a $B$-algebra. 
Then $R$ is commutative $\iff$ $B\otimes_Z {R}$ is commutative.
\end{lem}
\begin{proof}
{
Assume $B \otimes_Z R$ is commutative, and consider the ideal $I=R(r_1 r_2 - r_2 r_1)$ where $r_1,r_2 \in R$. Since $B \otimes_Z I=0$, we get $r_1r_2 -r_2 r_1=0$ by faithfully flatness.

Since $B$ is commutative, the other direction is clear.
}
\end{proof}
\begin{prop} [{A non-commutative variant of \cite[Proposition 2.10]{E95}}] \label{prop:commifftensoriscomm part 2}

Let $Z$ be a commutative ring, let $B$ be a flat commutative $Z$-algebra, and let $R$ be a Noetherian $Z$-algebra. 
Assume we are given $R$-modules $N_1$ and $N_2$ such that $N_1$ is finitely presented. Then the natural map 
\begin{gather*}
  \kappa:B \otimes_Z \Hom_{ R}(N_1,N_2)\to\Hom_{B \otimes_Z R}(B \otimes_Z N_1,B \otimes_Z N_2),\\
  \kappa(b \otimes_Z \varphi)(b' \otimes_Z n)=b b' \otimes_Z \varphi(n)
\end{gather*}
is an isomorphism. 
\begin{enumerate}
\item If furthermore $N_1\simeq N_2$, then $\kappa$ is an  isomorphism of algebras.
\item In particular, if $Z=R$ and $B=S^{-1}Z$ where $S \subset Z$ is multiplicatively closed  then flatness of $S^{-1}Z$  over $Z$ implies $\kappa$ induces an isomorphism
\[
\kappa:S^{-1}\Hom_{Z }(N_1, N_2) \xrightarrow{\sim} \Hom_{ S^{-1}Z}(S^{-1}N_1,S^{-1}N_2).
\]
\end{enumerate}
\end{prop}
%compare to Page 69, Prop 2.10 of Eisenbud's commutative algebra.
\begin{proof}
If $N_1 \simeq N_2$, this is clearly a map of rings.

We prove $\kappa$ is an isomorphism in two steps. 
Assume $N_1\simeq R^m$ is a free $R$-module.
We have 
\begin{gather*}
B\otimes_Z \Hom_{R}(N_1,N_2) =B\otimes_Z \Hom_{R}(R^m,N_2)  \simeq (B \otimes_Z N_2)^m, \\
\Hom_{B \otimes_Z R}(B \otimes_Z N_1,B \otimes_Z N_2) \simeq  \Hom_{B \otimes_Z R}((B \otimes_Z R)^m,B \otimes_Z N_2) \simeq(B \otimes_Z N_2)^m.
\end{gather*}
In this presentation, the map $\kappa$ is given on generators by $\kappa((b \otimes n) e_j)=(b \otimes n)e_j$, so it is an isomorphism.

For the general case, write a free resolution $P^\bullet$ of $N_1$: 
\[ P^\bullet \to N_1 \to 0 = \cdots \to P^2 \to P^1 \to P^0 \to N_1 \to 0 .\]
Since $B$ is flat over $Z$, 
tensoring preserves cohomologies:
\[
B\otimes_Z\Ext^i_{R}( N_1,N_2) 
=B\otimes_Z H^i (\Hom_{R}(P^\bullet,N_2)) 
= H^i (B\otimes_Z \Hom_{R}(P^\bullet,N_2)). 
\]
Consider the chain map $\kappa^\bullet:B\otimes_Z \Hom_{R}(P^\bullet,N_2) \to \Hom_{B \otimes_Z {R}} ( B \otimes_Z P^\bullet,B \otimes_Z N_2)$.
Each $P^i$ is free, so by Step 1 the maps $\kappa^i:B\otimes_Z \Hom_{R}(P^i,N_2) \to \Hom_{B \otimes_Z {R}} ( B \otimes_Z P^i,B \otimes_Z N_2) $ are isomorphisms and thus so is $\kappa^\bullet$.
Since $B\otimes_Z(-)$ is exact $B\otimes_Z P^\bullet$ is a free resolution of 
$B\otimes_Z N_1$. We get,
\begin{flalign*}
H^i (B\otimes_Z \Hom_{R}(P^\bullet,N_2))&\simeq H^i ( \Hom_{B\otimes_Z {R}}(B\otimes_Z P^\bullet,B\otimes_Z N_2))\\
&=\Ext^i_{B\otimes_Z {R}}(B\otimes_Z N_1,B\otimes_Z N_2).
\end{flalign*}
In particular, setting $i=0$ we see that $\kappa^0$ induces the desired isomorphism $\kappa$.
\end{proof}
\begin{proof}[Proof of Proposition \ref{prop:commifftensoriscomm}]
By Theorem \ref{thm:RamiEytanresults} the map $\mathrm{Spec} (B) \to \mathrm{Spec} (Z)$ 
is surjective and \'{e}tale, so it is faithfully flat. By Lemma \ref{lem:commifftensoriscomm part 1} we have
$\End_{R_\rho} (N)$ is commutative $\iff$ $B \otimes_{Z_\rho} \End_{ {R_\rho}} ( N)$ is commutative.
Using Proposition \ref{prop:commifftensoriscomm part 2} we get $B \otimes_{Z_\rho} \End_{ {R_\rho}}  N \simeq  \End_{ B \otimes_{Z_\rho}{R_\rho}}  (B \otimes_{Z_\rho}N)$ and the claim follows.
\end{proof}

For the proofs of the next two statements 
set $M=\ind_H^G \chi^{-1}\delta$.  
\begin{prop} \label{prop:commifftensorcomm}
$\mathrm{End}_{R_\rho}(H_{\rho,\ind_H^G \chi^{-1} \delta})$ is commutative $\iff$  $\End_B(H_{\rho,\ind_H^G \chi^{-1} \delta}|_B)$ is commutative.
\end{prop}
\begin{proof}
Assume $\mathrm{End}_{R_\rho}(H_{\rho,M})$ is commutative. 
By Theorem \ref{thm:RamiEytanresults} we have,
\[
\left( B \otimes_{Z_\rho} Z_\rho \subseteq  B \otimes_{Z_\rho} B\subseteq  B \otimes_{Z_\rho} R_\rho\right)
\simeq
\left(B\cdot \mathrm{Id} \subseteq \mathrm{diag}(B) \subseteq \mathrm{M}_n(B)\right).
\]
Set $N = B\otimes_{Z_\rho} H_{\rho,M}$.
It is an $\mathrm{M}_n(B)$-module, and by Morita equivalence we have
$N\simeq \tilde N^n$ where $\tilde N$ is a $B$-module and $\End_B(\tilde N)\simeq \End_{\mathrm{M}_n(B)}(N)$.

Since by Proposition \ref{prop:commifftensoriscomm} the algebra $\mathrm{End}_{B\otimes_{Z_\rho} R_\rho}(N)$ is commutative, we conclude that the algebra $\End_{B\otimes_{Z_\rho} B}(B\otimes_{Z_\rho} H_{\rho,M})$ is commutative as well:
\[
\End_{B\otimes_{Z_\rho} B}(B\otimes_{Z_\rho} H_{\rho,M})=\End_{B^n}(N)=
\End_{B^n}(\tilde{N}^n)\simeq(\End_B(\tilde N))^n.
\]
By applying  Proposition \ref{prop:commifftensoriscomm} once again, we see that $\End_{B}(H_{\rho,M}|_B)$
is commutative.

The converse is clear since $B \subset R_\rho$ implies $\End_{R_\rho}(H_{\rho,M})\subseteq \End_{B}(H_{\rho,M}|_B)$.
\end{proof}

Recall $(G,H):=({\bf G}(F),{\bf H}(F))$ is an $F$-spherical pair if
 $|H\backslash G/P|$ is finite for every $P={\bf P}(F)$ where ${\bf P}$ is a parabolic of ${\bf G}$.
Theorem \ref{thm:MainTheorem} now clearly follows from Theorem \ref{thm:commHeckeImpliesCGP} by ranging over all irreducible cuspidal representations $\rho$.
\begin{thm} \label{thm:commHeckeImpliesCGP}
Let $(G,H)$ be an $F$-spherical pair, assume 
$d_{H,\chi\delta}(\rho) < \infty$ and denote by $(\H_\chi)_\rho:=\End_{R_\rho}(H_{\rho,\ind_H^G \chi^{-1} \delta})$ the projection of $\H_\chi$ to the block $\mathcal{B}_\rho$.
Then 
\[(\H_\chi)_\rho\text{ is commutative }\iff
d_{H,\chi\delta}(\rho') \leq 1\text{ for every irreducible representation }\rho' \in \mathcal{B}_\rho.\]
\end{thm}
%Note generic flatness is not enough here- it might just say that the module is zero on an open set.
\begin{proof} 
Assume $H_{\rho,M}\neq 0$, as otherwise we are done.
Using Theorem \ref{thm:locally free sheaf}(2) we have $H_{\rho,M}=i_*(\mathcal{F})$ where $i: \frak{X}' \subset \frak{X}_G$ is smooth and $\mathcal{F}$ is locally free. 
Since $\mathcal{F}$ is locally free, $\supp(\mathcal{F})$ is a union of irreducibility components of $\frak{X}'$, so we can assume $\frak{X}'$ is the vanishing set of $I=\mathrm{Ann}_{B}(H_{\rho,M})$, the annihilator of $H_{\rho,M}$ in $B$.
%Eisenbud Corollary 2.7

We get that $H_{\rho,M}$ is locally free as a $B/I$-module,
so  there exist generators $f_1, \ldots, f_n \in B/I$ such that $f_i^{-1}H_{\rho,M}\simeq (f_i^{-1}(B/I))^{k_i}$ for some $k_i \in \nats$ and $\bigcup\limits_{i=1}^n D(f_i)=\mathrm{Spec}(B/I)$.

If $\End_{R_\rho}(H_{\rho,M})$ is commutative, then by Proposition \ref{prop:commifftensorcomm} so is $\End_{B/I}(H_{\rho,M})$.
By Corollary B.4.6 of \cite[Appendix B]{AGS15} (see also \cite{AAG12}), $H_{\rho,M}$ is finitely generated over $B$,
so it is a finitely presented $B/I$ module ($B/I$ is Noetherian).
Now, Proposition \ref{prop:commifftensoriscomm part 2}(2) implies 
$\End_{f_i^{-1}B/I}(f_i^{-1}H_{\rho,M}) \simeq f_i^{-1}\End_{B/I}(H_{\rho,M})$ is commutative.
We conclude that for each $i$ we have $f_i^{-1}H_{\rho,M}\simeq f_i^{-1}(B/I)$.

By Theorem \ref{thm:locally free sheaf}(1), $ \Hom_{G}(M,\rho\otimes \psi) \simeq \Hom_{\mathcal{O}_{\frak{X}_G}}(H_{\rho,M},\delta_{\psi})$, where $\delta_\psi$ is the skyscraper sheaf at $\psi \in\frak{X}_G$ with ring $\mathcal{O}_{\frak{X}_G,\psi}/\frak{m}_\psi$. 
Let $(H_{\rho,M})_\psi$ be the stalk of $H_{\rho,M}$ at $\psi$. We have,
\[
\Hom_{\mathcal{O}_{\frak{X}_G}}(H_{\rho,M},\delta_{\psi})
\simeq \Hom_{\mathcal{O}_{\frak{X}_G,\psi}}((H_{\rho,M})_\psi,\mathcal{O}_{\frak{X}_G,\psi}/\frak{m}_\psi).
\]
Since $H_{\rho,M}$ is locally {free} of rank $1$ over $\frak{X}'$ (and vanishes over the complement of  $\frak{X}'$),
 in particular $(H_{\rho,M})_\psi\simeq (\mathcal{O}_{\frak{X}_G,\psi}/\frak{m}_\psi)^k$ where $k \leq 1$. We get $d_{H,\chi\delta}(\rho\otimes \psi)= \dim_\complex \Hom_{G}(M,\rho\otimes \psi) \leq 1$:
\begin{flalign}
\dim_\complex \Hom_{\mathcal{O}_{\frak{X}_G}}(H_{\rho,M},\delta_{\psi})
&=\dim_\complex\Hom_{\mathcal{O}_{\frak{X}_G,\psi}}((H_{\rho,M})_\psi,\mathcal{O}_{\frak{X}_G,\psi}/\frak{m}_\psi) \notag
=\dim_\complex\Hom_{\mathcal{O}_{\frak{X}_G,\psi}}(\mathcal{O}_{\frak{X}_G,\psi}^k,\mathcal{O}_{\frak{X}_G,\psi}/\frak{m}_\psi)\leq 1.
\end{flalign}
Since every irreducible $\rho'  \in \mathcal{B}_\rho$ is of the form $\rho' \simeq \rho \otimes \psi$ for some $\psi \in \frak{X}_G$, we are done.

Conversely, we get the following for every $\psi \in \frak{X}_G$:
\[
\dim_\complex\Hom_{\mathcal{O}_{\frak{X}_G,\psi}}((H_{\rho,M})_\psi,\mathcal{O}_{\frak{X}_G,\psi}/\frak{m}_\psi)
=\dim_\complex \Hom_{\mathcal{O}_{\frak{X}_G}}(H_{\rho,M},\delta_{\psi}) =d_{H,\chi\delta}(\rho \otimes \psi)\leq 1.
\]
Recalling the locally freeness of $H_{\rho,M}$ over $\frak{X'}$, the rank of every $f_i^{-1}H_{\rho,M}$ is at most $1$.
It follows every
$f_i^{-1}\End_{B/I}(H_{\rho,M})
\simeq \End_{f_i^{-1}(B/I)}(f_i^{-1}H_{\rho,M})$
is commutative.
 Since $\curly{f_i}_{i=1}^n$ generate $B/I$, the map 
 \[
\End_{B/I}(H_{\rho,M}) \to f_1^{-1}\End_{B/I}(H_{\rho,M}) \times \ldots \times f_n^{-1}\End_{B/I}(H_{\rho,M})
 \]
is injective, and we get that $\End_{B/I}(H_{\rho,M})$ is commutative. Proposition \ref{prop:commifftensorcomm} finishes the proof.
\end{proof}
\section{A tempered $\chi$-Gelfand property implies $\H_\chi$ is commutative}
Let $H \leq G$ be l.c.t.d groups as in Section \ref{sec:conventions}.
Throughout this section, we assume $G$ is second countable of type I (see Section \ref{sec:Direct integral decompositions of unitary representations}) and that $\chi$ is a unitary character.
Let $W(\chi):=L^2(H \backslash G,\chi^{-1}\delta)$ be the completion of the unitarizable representation $\ind_H^G \chi^{-1} \delta$,
%
%:=L^2(H \backslash G,\chi^{-1}\delta)$ be the space of all square-integrable sections of the bundle of half-densities over $H \backslash G$, twisted by $\chi^{-1}$ 
and recall $\rho \in \Irr{G}$ is said to be $(H\backslash G,\chi^{-1})$-{tempered} if it is included in the support of the Plancherel measure  of $W(\chi)$ (see Section \ref{sec:Direct integral decompositions of unitary representations}).
\begin{defn}
We say that $(G,H)$ is an $(H\backslash G,\chi^{-1})$-{tempered} $\chi$-Gelfand pair if 
$d_{H,\chi\delta} \leq 1$ for every $(H\backslash G,\chi^{-1})$-{tempered} $\rho \in \Irr{G}$.
\end{defn}

In this section we prove the following theorem.
\begin{thm}\label{thm:GPimpliesCom}
Let $(G,H)$ be an $(H\backslash G,\chi^{-1})$-tempered $\chi$-Gelfand pair, then $\mathcal{H}_\chi$ is commutative.
\end{thm}

Given the Hecke algebra $A(G)$ of locally constant, compactly supported measures on $G$, a standard result (\cite[Separation Lemma]{BR})
 asserts that for any $0 \neq a \in A(G)$ there exists $\rho\in \mathrm{Irr}(G)$ such that $\rho(h)\neq 0$.
One then says that $A(G)$ is separated by the set of its irreducible representations.
We now formulate and prove an analogous statement for $\ind_H^G \chi^{-1}\delta$ from
which, after an additional step, Theorem \ref{thm:GPimpliesCom} will follow.
\begin{defn}
Let $M$ be a $G$-representation, and let $\mathrm{rad}(M)$ be the intersection of its maximal sub-representations.
\begin{enumerate}
\item We say that $M$ is  semi-primitive or separated if $\mathrm{rad}(M) =  \curly{0}$.
\item We say that $M$ is separated by a set $\curly{(M_\gamma,p_\gamma: M \to M_\gamma)}_{\gamma \in I}$ 
if every $M_\gamma$ is an irreducible $G$-representation and for every $m \in M$ there exists $\gamma \in I$
such $p_\gamma(m) \neq 0$.
\end{enumerate}
\end{defn}
Let $\nu_{W(\chi)}$ be the Plancherel measure of $W(\chi):=L^2(H \backslash G,\chi^{-1}\delta)$, denote by $\rho^{\mathrm{sm}}$ the smooth vectors of $\rho$ 
and set $S_{\nu_{W(\chi)}}:=\curly{\rho^{\mathrm{sm}}_\gamma: \exists \alpha_\gamma:L^2(H \backslash G,\chi^{-1}\delta)\to\rho_\gamma \text{ and } (\rho_\gamma,\alpha_\gamma)\in \supp(\nu_{W(\chi)})}.$
\begin{prop} \label{prop:indisseparated}
The module $\ind_H^G \chi^{-1}\delta$ is separated by the set
$\curly{(\rho^{\mathrm{sm}}_\gamma,\alpha_\gamma)}_{(\rho_\gamma,\alpha_\gamma)\in \supp(\nu_{W(\chi)})}$.
\end{prop}
\begin{proof}
{Consider the central decomposition of $W(\chi)$
\[W(\chi):=L^2(H \backslash G,\chi^{-1}\delta)= \int^\oplus_{\widehat{G}}m_{\nu_{W(\chi)}}(\rho) \rho d\nu_{W(\chi)}(\rho).\]
By Proposition \ref{prop:Bernstein separation}, the embedding $\alpha: \ind_H^G \chi^{-1}\delta \to W(\chi)$
is pointwise defined. 
This implies we have a family of $G$-maps $\curly{\alpha_\gamma : \ind_H^G \chi^{-1} \delta \to \rho_\gamma^{\mathrm{sm}}}_{\gamma \in I} $ where $\rho_\gamma \in \widehat{G}$,
such that for every $f \in \ind_H^G \chi^{-1}\delta$ we have $\alpha(f)(\gamma) = \alpha_\gamma(f)$.
In particular $\curly{(\rho_\gamma,\alpha_\gamma): \alpha_\gamma(f) \neq 0}$ has positive measure for every $f \in \ind_H^G \chi^{-1}\delta$. }
\end{proof}

The following is the final ingredient needed to prove Theorem \ref{thm:GPimpliesCom}:

\begin{prop} \label{prop:HeckeInjMap}
The map 
$\phi_\chi : \mathcal{H}_\chi \to \prod\limits_{\rho \in S_{\nu_{W(\chi)}}} 
\End_\complex( \Hom_G(\ind_H^G \chi^{-1}\delta,\rho))$ 
via 
\[\varphi \mapsto \phi_\chi(\varphi)\left((\alpha_\rho)_{\rho \in S_{\nu_{W(\chi)}}}\right)
=(\alpha_\rho \circ \varphi)_{\rho \in S_{\nu_{W(\chi)}}}\] 
 is an injective map of rings.
\end{prop}
\begin{proof} 
This is clearly a map of rings. 
Given a non-zero $\varphi \in \mathcal{H}_\chi$, there exists $f$ such that $\varphi(f) \neq 0$, and by the previous proposition there exist
an irreducible representation $\rho_\gamma \in S_{\nu_{W(\chi)}}$ and $\alpha_\gamma:\ind_H^G \chi^{-1} \delta \twoheadrightarrow \rho_\gamma$ such that 
$\alpha_\gamma(\varphi(f)) \neq 0$, as required.
\end{proof}
Theorem \ref{thm:GPimpliesCom} now follows immediately:
\begin{proof}[ {Proof of Theorem \ref{thm:GPimpliesCom}}]
If $(G,H)$ is an $(H\backslash G,\chi^{-1})$-tempered $\chi$-Gelfand pair, then 
\[\prod\limits_{\rho \in S_{\nu_{W(\chi)}}} \End_\complex(\Hom_G(\ind_H^G \chi^{-1} \delta,\rho))\] 
is a product of one dimensional algebras. 
Since $\phi$ is injective, $\mathcal{H}_\chi$ must be commutative.
\end{proof}

\appendix
\section{A general proof of the Gelfand-Kazhdan criterion}
\label{app: general proof of GK}
In this appendix we prove Theorem \ref{prop:GK} in the case of a locally compact, totally disconnected, second countable topological group $G$ with respect to a character $\chi$ of $H \leq G$
where we do not assume either $G$ or $H$ are unimodular.
Let $A(G)$ be the algebra of smooth, compactly supported measures on $G$, and let $C^{-\infty}(G)$ be its dual space, the space of generalized functions on $G$.
\begin{thm} [Gelfand-Kazhdan criterion] 
Assume there exists an anti-involution $\sigma:G\to G$ such that $\sigma(H)=H$ and $\sigma(\xi)=\xi$ for every generalized function $\xi \in C^{-\infty}(G)^{(H,\delta_G\chi^{-1}) \times (H,\delta_G\chi^\sigma)}$.
Then we have
\[
\dim_\complex\Hom_H(\rho, \chi) \cdot \dim_\complex\Hom_H(\tilde \rho, {\chi^\mu}) \leq 1,
\]
for every smooth irreducible admissible representation $\rho$ of $G$, where $\mu(g)=\sigma(g)^{-1}$.
\end{thm} 
\begin{proof}
We essentially follow \cite[Proposition 4.2]{Gr91} (see also \cite[Lemma 4.2]{P90}), using Lemma \ref{lem:inducedinvolutiononGxG} to replace a few arguments.
 
 Let $\rho$ be an irreducible representation
of $G$ and let $l : \rho \to \complex_{\chi}$ and $m : \tilde{\rho} \to \complex_{\chi^\mu} $ be non-zero 
$H$-quasi-invariant linear functionals. 
Choose a right invariant Haar measure $d_r g$ on $G$.
Then the functionals above give surjective $G$-linear maps $F_l : A(G) \to \tilde{\rho}$
and $F_m : A(G) \to \tilde{\tilde{\rho}} \simeq \rho$ defined by
\[
F_l(f)=\int_G \angles{\rho^*(g)l,-}df(g)=\rho^*(f)(l) \in \tilde{\rho}, 
~\text{        }~ F_m(f)=\int_G \angles{\tilde{\rho}^{*}(g)m,-}df(g)=\tilde{\rho}^{*}(f)(m) \in \rho
\]
for $f \in A(G)$.
Since $\rho$ and $\tilde{\rho}$ are irreducible,
the linear maps $F_l$ and $F_m$ are determined (up to a scalar) by their kernels (note we use here Schur's lemma).
Composing with the $G$-invariant bilinear form $\angles{-,-} : \rho \times \tilde{\rho} \to \complex $
we obtain a linear map, $B(f_1,f_2)=\angles {F_m(f_1), F_l(f_2)}$
\[ B : A(G) \otimes A(G) \to \rho \otimes \tilde{\rho} \to \complex. \]

We now claim $B$ may be viewed as a generalized function on $G \times G$ which is left
equivariant under $(H,\delta_G\chi^\sigma) \times (H,\delta_G\chi^{-1})$ and right invariant under the diagonal action of $G$. We calculate,
\begin{flalign*}
\angles{L_{(h_1,h_2)}B, f_1 \otimes f_2}
=B(L_{h_1^{-1}}f_1,L_{h_2^{-1}}f_2)
= \angles{F_m(L_{h_1^{-1}}f_1),F_l(L_{h_2^{-1}}f_2)}.
\end{flalign*}
Now, recall that $m$ is $(H,\chi^\mu)$-equivariant,
\begin{flalign*}
F_m(L_{h_1^{-1}}f_1)(\tilde{v})
&=\int_G \angles{m,\tilde{\rho}(g)\tilde{v}}dL_{h_1^{-1}}f_1(g)
=\int_G \delta_G(h_1)\angles{m,\tilde{\rho}(h_1^{-1})\tilde{\rho}(g)\tilde{v}}df_1( g) \\
&=\delta_G(h_1)\int_G \chi^\mu(h_1^{-1})\angles{m,\tilde{\rho}(g)\tilde{v}}d f_1( g)
=\delta_G(h_1)\chi^\sigma(h_1)F_m(f_1)(\tilde{v}).
\end{flalign*}
Similarly, $F_l(L_{h_2^{-1}}f_2)=\delta_G(h_2)\chi^{-1}(h_2)F_l(f_2)$.
Since $F_m$ and $F_l$ are $G$-linear (recall our measure is $G$-invariant on the right), $F_m(R_{g'}f)=\rho(g')( F_m(f))$, and similarly for $F_l$.
Since our bilinear form is $G$-invariant, we get that $B$ is invariant with respect to right $\Delta G$ action:
\begin{flalign*}
\angles{R_{(g,g)}B, f_1 \otimes f_2}
= \angles{F_m(R_{g^{-1}}f_1),F_l(R_{g^{-1}}f_2)}
= \angles{\tilde{\rho}(g) F_m(f_1),\rho(g) F_l(f_2)}
= \angles{F_m(f_1), F_l(f_2)}.
\end{flalign*} 

Now, $B \in C^{-\infty}(G \times G)^{(H,\delta_G\chi^{-1}) \times (H,\delta_G\chi^\sigma)\times \Delta G}$,
 and by Lemma \ref{lem:inducedinvolutiononGxG} the isomorphism $C^{-\infty}(G) \simeq C^{-\infty}(G \times G)^{\Delta G}$ 
 induces the involution $(f_1 \otimes f_2)^\sigma = f_2^\mu \otimes f_1^\mu$ on $C^{-\infty}(G \times G)^{\Delta G}$.
By our assumptions $B$ is invariant under $\sigma$. In particular, 
\[
\angles{F_m(f_1),F_l(f_2)}=B(f_1, f_2) =B^\sigma(f_1,f_2)= B(f_2^\mu,f_1^\mu)=\angles{F_m(f_2^\mu),F_l(f_1^\mu)}
\]
for all $f_1, f_2\in A(G)$. 
This means $f_1 \in \ker F_m \iff f_1^\mu \in \ker F_l$ (we can fix $f_1$ and let $f_2$ vary), 
so $F_l$ determines the kernel of $F_m$, and since $\rho$ is irreducible it determines $F_m$ up to a scalar.
Since $l$ was arbitrary, we deduce that $\dim_\complex ( \tilde{\rho}, \chi^\mu) \leq 1$, and similarly
$\dim_\complex ( \rho, \chi) \leq 1$.
\end{proof}


\begin{thebibliography}{0}

\bibitem[AAG12]{AAG12}
Aizenbud, A., N. Avni  and D. Gourevitch, \emph{Spherical pairs over
  close local fields}, Comment. Math. Helv. \textbf{87} (2012), no.~4,
  929--962. 

\bibitem[AG09]{AG09}
Aizenbud, A. and D. Gourevitch, \emph{Generalized {H}arish-{C}handra
  descent, {G}elfand pairs, and an {A}rchimedean analog of {J}acquet-{R}allis's
  theorem}, Duke Math. J. \textbf{149} (2009), no.~3, 509--567, with an
  appendix by the authors and Eitan Sayag. 

\bibitem[AGRS10]{AGRS10}
Aizenbud, A., D. Gourevitch, S. Rallis, and G. Schiffmann,
  \emph{Multiplicity one theorems}, Ann. of Math. (2) \textbf{172} (2010),
  no.~2, 1407--1434. 

\bibitem[AGS08]{AGS08}
Aizenbud, A., D. Gourevitch and E. Sayag, \emph{{$({\rm
  GL}_{n+1}(F),{\rm GL}_n(F))$} is a {G}elfand pair for any local field {$F$}},
  Compos. Math. \textbf{144} (2008), no.~6, 1504--1524. 

\bibitem[AGS15]{AGS15}
Aizenbud, A., D. Gourevitch and E. Sayag, \emph{{$\frak{Z}$}-finite distributions on {$p$}-adic groups}, Adv.
  Math. \textbf{285} (2015), 1376--1414. 

\bibitem[AS20]{AS}
Aizenbud, A. and E. Sayag, \emph{Homological multiplicities in representation theory of $p$-adic groups},  Math. Z. 294 (2020), no. 1-2, 451–469.

\bibitem[Ber74]{Ber3}
Bern{\v{s}}te{\u\i}n I.~N., \emph{All reductive {${\frak p}$}-adic groups are
  of type {I}}, Funkcional. Anal. i Prilo\v zen. \textbf{8} (1974), no.~2,
  3--6. 

\bibitem[Ber88]{Ber2}
Bernstein J. N., \emph{On the support of {P}lancherel measure}, J. Geom.
  Phys. \textbf{5} (1988), no.~4, 663--710 (1989). 
  
\bibitem[BD84]{BD84}
{Bernstein, J. N.}, \emph{Le ``centre'' de {B}ernstein},
{Representations of reductive groups over a local field},
{Travaux en Cours}, {Hermann, Paris}, 1984, {Edited by P. Deligne}, pp. 1-32.
  
\bibitem[BR]{BR}
Bernstein J. N. and K. Rumelhart, \emph{Lectures on $p$-adic groups}, 
{Unpublished}, 
{available at http://www.math.tau.ac.il/$\sim$bernstei/Publication\_list/publication\_texts/Bernst\_Lecture\_p-adic\_repr.pdf}.

\bibitem[BZ76]{BZ}
Bern{\v{s}}te{\u\i}n I.~N.  and A.~V. Zelevinski{\u\i}, \emph{Representations of
  the group {$GL(n,F),$} where {$F$} is a local non-{A}rchimedean field},
  Uspehi Mat. Nauk \textbf{31} (1976), no.~3(189), 5--70. 

\bibitem[BC95]{BC95}
Bost J.-B.  and A. Connes, \emph{Hecke algebras, type {III} factors and phase
  transitions with spontaneous symmetry breaking in number theory}, Selecta
  Math. (N.S.) \textbf{1} (1995), no.~3, 411--457. 

\bibitem[Cog04]{Co04}
James~W. Cogdell, \emph{Lectures on {$L$}-functions, converse theorems, and
  functoriality for {${\rm GL}_n$}}, Lectures on automorphic {$L$}-functions,
  Fields Inst. Monogr., vol.~20, Amer. Math. Soc., Providence, RI, 2004,
  pp.~1--96.

\bibitem[Del10]{Del10}
Delorme P., \emph{Constant term of smooth {$H_\psi$}-spherical functions
  on a reductive {$p$}-adic group}, Trans. Amer. Math. Soc. \textbf{362}
  (2010), no.~2, 933--955. 
  
\bibitem[Dia88]{Di88}
Diaconis P., \emph{Group representations in probability and statistics},
  Institute of Mathematical Statistics Lecture Notes---Monograph Series,
  vol.~11, Institute of Mathematical Statistics, Hayward, CA, 1988. 
  
\bibitem[Eis95]{E95}
Eisenbud D., \emph{Commutative algebra}, Graduate Texts in Mathematics, vol.
  150, Springer-Verlag, New York, 1995, With a view toward algebraic geometry.


\bibitem[F{\"u}h05]{Fuhr}
F{\"u}hr H., \emph{Abstract harmonic analysis of continuous wavelet
  transforms}, Lecture Notes in Mathematics, vol. 1863, Springer-Verlag,
  Berlin, 2005. 

\bibitem[GK75]{GK75}
Gel'fand I.~M. and D.~A. Kajdan, \emph{Representations of the group {${\rm
  GL}(n,K)$} where {$K$} is a local field}, 95--118. 

\bibitem[GPSR97]{GPSR97}
Ginzburg D., I. Piatetski-Shapiro, and S. Rallis, \emph{{$L$} functions for the
  orthogonal group}, Mem. Amer. Math. Soc. \textbf{128} (1997), no.~611,
  viii+218. 

\bibitem[Gro91]{Gr91}
Gross B. H., \emph{Some applications of {G}elfand pairs to number
  theory}, Bull. Amer. Math. Soc. (N.S.) \textbf{24} (1991), no.~2, 277--301.


\bibitem[Hak03]{Hak03}
Hakim J., \emph{Supercuspidal {G}elfand pairs}, J. Number Theory
  \textbf{100} (2003), no.~2, 251--269. 

\bibitem[HM08]{HM08}
Hakim J. and F. Murnaghan, \emph{Distinguished tame supercuspidal
  representations}, Int. Math. Res. Pap. IMRP (2008), no.~2, Art. ID rpn005,
  166. 

\bibitem[KLQ08]{KLQ08}
Kaliszewski S., M. B. Landstad, and J. Quigg, \emph{Hecke
  {$C^*$}-algebras, {S}chlichting completions and {M}orita equivalence}, Proc.
  Edinb. Math. Soc. (2) \textbf{51} (2008), no.~3, 657--695.

\bibitem[KS18]{KS18}
Kr\"otz B. and H. Schlichtkrull, \emph{Harmonic analysis for real
  spherical spaces}, Acta Math. Sin. (Engl. Ser.) \textbf{34} (2018), no.~3,
  341--370.

\bibitem[LB16]{LB16}
Le Boudec, Adrien, \emph{Groups acting on trees with almost prescribed local action}, 
Comment. Math. Helv., {\bf 91} (2016), no.~2, 253--293.

\bibitem[Let82]{Le82}
Letac G., \emph{Les fonctions sph\'{e}riques d'un couple de {G}el'fand
  sym\'etrique et les cha\^{i}nes de {M}arkov}, Adv. in Appl. Probab.
  \textbf{14} (1982), no.~2, 272--294.

\bibitem[Off11]{Of}
Offen O., \emph{On local root numbers and distinction}, J. Reine Angew. Math.
  \textbf{652} (2011), 165--205.

\bibitem[OV96]{OV96}
Okounkov A. and Vershik A., \emph{A new approach to representation theory of symmetric groups}, 
Selecta Math. (N.S.) {\bf 2} (1996), no. 4, 581-605.

\bibitem[Ol'77]{Ol77}
G.~I. Ol'\v{s}anski\u{\i}, \emph{Classification of the irreducible
  representations of the automorphism groups of {B}ruhat-{T}its trees},
  Funkcional. Anal. i Prilo\v{z}en. \textbf{11} (1977), no.~1, 32--42, 96.

\bibitem[PR94]{PR94}
Platonov V. and A. Rapinchuk, \emph{Algebraic groups and number
  theory}, Pure and Applied Mathematics, vol. 139, Academic Press, Inc.,
  Boston, MA, 1994, Translated from the 1991 Russian original by Rachel Rowen.

\bibitem[Pra90]{P90}
Prasad D., \emph{Trilinear forms for representations of {${\rm GL}(2)$}
  and local {$\epsilon$}-factors}, Compositio Math. \textbf{75} (1990), no.~1,
  1--46. 

\bibitem[Rez08]{Re08}
Reznikov A., \emph{Rankin-{S}elberg without unfolding and bounds for
  spherical {F}ourier coefficients of {M}aass forms}, J. Amer. Math. Soc.
  \textbf{21} (2008), no.~2, 439--477. 

\bibitem[RW]{RW}
Reid C. D. and P.~R. Wesolek, \emph{Homomorphisms into totally
  disconnected, locally compact groups with dense image.},  Forum Math. 31 (2019), no. 3, 685–701.

\bibitem[Ro09]{Ro09}
Roche A., \emph{The {B}ernstein decomposition and the {B}ernstein centre},
{Ottawa lectures on admissible representations of reductive {$p$}-adic groups}, 
{Fields Inst. Monogr.}, vol. 26, {Amer. Math. Soc., Providence, RI}, 2009, pp. 3-52.

\bibitem[SV17]{SV17}
Sakellaridis Y. and A. Venkatesh, \emph{Periods and harmonic analysis
  on spherical varieties}, Ast\'erisque (2017), no.~396, viii+360. 

\bibitem[Sha74]{Sha74}
Shalika J.A., \emph{The multiplicity one theorem for {${\rm GL}_{n}$}}, Ann.
  of Math. (2) \textbf{100} (1974), 171--193.

\bibitem[SZ11]{SZ11}
Sun B. and C.-B. Zhu, \emph{A general form of {G}elfand-{K}azhdan
  criterion}, Manuscripta Math. \textbf{136} (2011), no.~1-2, 185--197.
  
\bibitem[vD09]{vD09}
van Dijk, Gerrit, \emph{Introduction to harmonic analysis and generalized {G}elfand pairs}, 
{De Gruyter Studies in Mathematics}, vol. 36, {Walter de Gruyter \& Co., Berlin}, 2009.

\end{thebibliography}
\end{document}